\documentclass[11pt,a4paper]{article}
\pdfoutput=1

\usepackage[a4paper,left=2.0cm,right=2.0cm,top=3cm,bottom=3cm]{geometry}
\usepackage{amsthm,amsmath,amssymb,amsfonts}
\usepackage{braket}
\usepackage[hidelinks]{hyperref}
\usepackage{mathtools}

\usepackage{authblk}

\newtheorem{thm}{Theorem}[subsection]

\theoremstyle{definition}
\newtheorem{cor}[thm]{Corollary}
\newtheorem{dfn}[thm]{Definition}
\newtheorem{lem}[thm]{Lemma}

\theoremstyle{remark}
\newtheorem{calc}{\underline{Calculation}}[subsection]

\newtheorem{result}{Result}
\newtheorem{op}{Open Problem}
\newtheorem*{opAlt1}{Open Problem 1 (Alternate Version)}
\newtheorem*{opAlt2}{Open Problem 2 (Alternate Version)}

\def\tp{\otimes}

\def\R{\mathbb{R}}
\def\C{\mathbb{C}}

\newif\ifPrivateMode
\PrivateModefalse

\title{Temperley-Lieb and Birman-Murakami-Wenzl like relations from multiplicity free semi-simple tensor systems}
\author{Peter E. Finch}\affil{Institut f\"ur Theoretische Physik, Leibniz Universit\"at Hannover,
Appelstra\ss{}e 2, 30167 Hannover, Germany}
\date{\today}

\begin{document}
\maketitle


\begin{abstract}
\noindent
In this article we consider conditions under which projection operators in multiplicity free semi-simple tensor categories satisfy Temperley-Lieb like relations. This is then used as a stepping stone to prove sufficient conditions for obtaining a representation of the Birman-Murakami-Wenzl algebra from a braided multiplicity free semi-simple tensor category. The results are found by utalising the data of the categories. There is considerable overlap with the results found in \cite{FiKaMa2016}, where proofs are shown by manipulating diagrams.
\end{abstract}

\section{Introduction}
The Temperley-Lieb algebra is a diagram algebra which first appeared in the context of two-dimensional statistical mechanics \cite{TemLie1971}. Its importance in one-dimensional quantum physics and two-dimensional statistical mechanical models has not diminished as the identification of representations of the algebra allows the identification of integrable points. This subsequently allows the application of various Bethe ans{\"a}tze in order to compute physical quantities exactly. One can also apply use the equivalence of such models to the Heisenberg spin-1/2 XXZ model \cite{AufKlu2010}. There has also been substantial interest in the other aspects of the Temperley-Lieb algebra including its representations \cite{MartinBook1995,Westbu1995} and its connections to other fields such as the study of knot theory, quantum groups, and quantum information \cite{Abrams2009,BatKun1991,WaDeAk1989,ZKG2005}.

A key result in the study of representation theory of the Temperley-Lieb algebra is Schur-Weyl duality \cite{WeylBook1939} (Is this appropriate). Amongst other things, Schur-Weyl duality states that given $L$ copies of the spin-1/2 representation of $U_{q}(sl(2))$, representations of the generators of the Temperley-Lieb correspond to projection operators onto the trivial representation. This connection between representations or the TL generators and projection operators onto trivial representations has also been observed in many other cases, such as other quantum groups \cite{BatKun1991,KobLim1996,Reshet1988,Zhang1991} and Drinfeld doubles of finite group algebras \cite{Gould1993}.

This article follows these observations and determines sufficient and necessary conditions for certain projection operators to satisfy the Temperley-Lieb like relations. The results are expressed in terms of certain categories of representations, specifically semi-simple tensor categories, rather than explicit representations. Results are subsequently extended to Birman-Murakami-Wenzl (BMW) which contains a subalgebra isomorphic to the Temperley-Lieb algebra. When applicable one can also invoke Pasquier's face-vertex correspondence to apply the results to explicit representation \cite{Pasqui1988}.

\subsection{Outline of results}
The core results of this paper concern the construction of operators that satisfy Temperley-Lieb and Birman-Murakami-Wenzl like relations. The proofs presented are technical in nature and consequently to assist the reader we first state the results along with opens problems in simpler language. The first result presented gives sufficient and necessary conditions to obtain Temperley-Lieb like relations from two-site projection operators. 

\begin{center} \begin{minipage}{0.9\textwidth}
\vspace{0.2cm}
\begin{result} \label{ResGenTL}
In a multiplicity free semi-simple tensor category checking that two-site projection operators satisfy Temperley-Lieb like relations is equivalent to checking that a special set of gauge invariant quantities contains precisely one distinct number.
\end{result}
\vspace{0.2cm}
\end{minipage} \end{center}

\noindent
This results, while mathematically interesting, does not provide much insight into when projections operators satisfy Temperley-Lieb like relations. Furthermore, to apply the above result one needs to know the structure constants, known as F-moves or generalised 6-j symbol, of the semi-simple tensor category. Finding such constants is often a strenuous task. It is possible to give weaker statements, specifically, ones with more restrictive conditions, which are more useful.

\begin{center} \begin{minipage}{0.9\textwidth}
\vspace{0.2cm}
\begin{result} \label{ResPart1DTL}
Consider a  multiplicity free semi-simple tensor category and two sequences of simple objects, $\lambda_{1}\dots\lambda_{L}$ and $\nu_{1}\dots\nu_{L-1}$. If $\lambda_{i-1}\tp\nu_{i}$ and $\nu_{i}\tp\lambda_{i+2}$ are simple and $\lambda_{i}\tp\lambda_{i+1} \cong \nu_{i}\oplus\cdots$ then it is possible to construct operators that satisfy Temperley-Lieb like relations. The operators satisfying these relations correspond to two-site projection operators onto the simple objects $\nu_{i}$.
\end{result}

\begin{result} \label{Res1DTL}
In a multiplicity free semi-simple tensor category two-site projection operators onto one-dimensional objects satisfy Temperley-Lieb like relations.
N.B. An object is said to be one-dimensional if the tensor product with it and any other simple object is also simple.
\end{result}
\vspace{0.2cm}
\end{minipage} \end{center}

\noindent
To apply either Result 2 or 3 one only needs knowledge of the simple objects in the category and their tensor product decomposition. This is seen as a benefit as one can quickly see if there is guaranteed to be operators that satisfy Temperley-Lieb like relations. It is Result 3 that provides the simplest constraint with its implication that two-site projection operators onto one-dimensional objects satisfy Temperley-Lieb like relations.

Indeed the most well-known special cases of these results can be deduced from Result 3. For instance in the tensor product of $\mathcal{L}$ copies the spin-1/2 (two-dimensional) representation of $U_{q}(su(2))$ the two-site projection operator onto the trivial (a.k.a. spin-0 or one-dimensional) representation satisfies the Temperley-Lieb relations. This is precisely the aforementioned Schur-Weyl duality. There are many more examples that can be found throughout the literature, such as Pasquier's construction for simply laced Dynkin diagrams \cite{Pasquier1987a}, when the fusion category is the category of representations of a Drinfeld double of a finite group \cite{Gould1993} or certain quantum groups \cite{BatKun1991,KobLim1996,Reshet1988,Zhang1991}. However, each of these results are restricted in some form such as a restriction to a specific category, a restriction to sequence of only one type of object or a restriction to only considering projection operators onto the trivial object. The above results are more general as they allow non-uniform sequences of seed objects. In the study of integrable systems one sees that it allows the consideration of certain chains of alternating particles or chains with inhomogeneities.

Along side Result 1 we presented two additional results which are seemingly weaker. However, the author is not aware of any examples of semi-simple tensor categories that show that this is definitely the case. Thus we are lead to the following open problems.

\begin{center} \begin{minipage}{0.9\textwidth}
\vspace{0.2cm}
\begin{op} \label{op12}
Are there any multiplicity free semi-simple tensor categories where the conditions of Result 1 are met but the conditions of Result 2 are not? An alternative more precise version is given in the next section.
\end{op}

\begin{op} \label{op23}
Are there any multiplicity free semi-simple tensor categories where there exists simple objects $\lambda,\nu$ such that $\lambda\tp\nu$ is simple but neither $\lambda$ nor $\nu$ are one-dimensional (excluding trivial cases and cases which rely on one-dimensional objects in a disguised manner)?
\end{op}
\vspace{0.2cm}
\end{minipage} \end{center}

\noindent
Proving an answer to the first open problem would determine if knowledge of the fusion rules of the algebra provides the sufficient and necessary information to determine if a two-site projection operator satisfies Temperley-Lieb like relations. For the second open problem, we see that if the answer is no then one may consider Results 2 and 3 to be equivalent after some additional constraints. We remark in the second open problem we have excluded certain cases in a fairly vague manner. This is exclusion is to avoid cases which are deemed uninteresting and is discussed in more detail in the next section.

The Birman-Murakami-Wenzl algebra \cite{BirWen1989,Muraka1987} includes a sub-algebra isomorphic to the Temperley-Lieb algebra. As such it is possible to use the results for the Temperley-Lieb like relations as a foothold in finding operators that satisfy BMW like relations. In particular we find sufficient conditions for non-trivial operators to satisfy BMW like relations in a semi-simple tensor category which is also braided.

\begin{center} \begin{minipage}{0.9\textwidth}
\vspace{0.2cm}
\begin{result}
Consider a multiplicity free braided semi-simple tensor system and a seed simple object, $\lambda$. If $\lambda\tp\lambda\cong\nu\oplus\cdots$ and $\lambda\tp\nu$ is simple then it is possible construct operators such that all bar one of the BMW relations are satisfied. The final BMW relation can only be satisfied when the braiding operator has at most three distinct eigenvalues which satisfy a certain condition.
\end{result}
\vspace{0.2cm}
\end{minipage} \end{center}

\noindent
The results stated above, and proven in the next section, all require that the semi-simple tensor category is multiplicity free. However, there are many examples of semi-simple tensor categories which are not. It would be of interest to determine if the above results extend to these cases.

\begin{center} \begin{minipage}{0.9\textwidth}
\vspace{0.2cm}
\begin{op} \label{opWithMult}
Do these results extend to semi-simple tensor systems that are not multiplicity free?
\end{op}
\vspace{0.2cm}
\end{minipage} \end{center}

\noindent
The last open question considered concerns generalising the approach to find representations of the Hecke algebra.

\begin{center} \begin{minipage}{0.9\textwidth}
\vspace{0.2cm}
\begin{op} \label{opWithMult}
Can one find analogous sufficient or necessary conditions on object in a semi-simple tensor category that leads to projection operators satisfying the Hecke relations?
\end{op}
\vspace{0.2cm}
\end{minipage} \end{center}

\noindent
We argue this is a natural question to ask because the Temperley-Lieb algebra is a quotient of the Hecke algebra.

As stated previously, all results presented are expressed in terms of semi-simple tensor categories where the operators act on Hilbert spaces spanned by basis vectors given by fusion paths. Such Hilbert spaces are often associated with interaction round-a-face models and anyon chains \cite{GRS1996,Finch2013}. In the case when the semi-simple tensor category is also a category of matrix representations of a semi-simple Hopf algebra one can apply Pasquier's face-vertex correspondence to the operators \cite{Pasqui1988}. Thus there will be analogous operators given in terms Hilbert space associated with the matrix representations that satisfy equivalent equations. This second basis is often associated with vertex models and spin chains.

\subsection{Definitions}
Instead of using diagrammatic calculus to prove results we use the combinatorial datum and structure constants of the categories. Throughout $\textbf{R}$ will be a commutative ring with identity. Also, we will only consider semi-simple tensor categories which are multiplicity free. As such the reader should assume that every semi-simple tensor category is multiplicity free unless specifically stated otherwise.\footnote{Generally $N_{ab}^{c}\in \{0,1,2,\dots\}$, being multiplicity free requires that $N_{ab}^{c}\in \{0,1\}$. This also has implications on how the F-moves are presented \cite{BonderThesis2007}}
\begin{dfn}
A \textit{(multiplicity free) semi-simple tensor system} consists of a set $\mathcal{I}$, as well as constants $N_{ab}^{c}\in \{0,1\}$ and $\left(F^{abc}_{d}\right)^{e}_{f},\left(\bar{F}^{abc}_{d}\right)^{e}_{f}\in \textbf{R}$ for $a,b,c,d,e,f\in\mathcal{I}$. The $N$s satisfy the properties:
\begin{enumerate}
	\item [N.1] \textit{Associativity}: $\sum_{e\in\mathcal{I}}N_{ab}^{e}N_{ec}^{d}=\sum_{e\in\mathcal{I}}N_{ae}^{d}N_{bc}^{e}$ $\forall a,b,c,d\in\mathcal{I}$,
	\item [N.2] \textit{Finiteness}: $\sum_{c\in\mathcal{I}}N_{ab}^{c} < \infty$ $\forall a,b\in\mathcal{I}$.
\end{enumerate}
The F-moves satisfy the relations:
\begin{enumerate}
	\item [F.1] \textit{Zero Condition}: $N_{ab}^{e}N_{bc}^{f}N_{af}^{d}N_{ec}^{d} = 0$ implies $(F^{abc}_{d})^{e}_{f}= (\bar{F}^{abc}_{d})^{f}_{e} = 0$  $\forall a,b,c,d,e,f\in \mathcal{I}$,
	\item [F.2]\textit{Pentagon Relation}: $\sum_{h\in\mathcal{I}} \left(F^{abc}_{g}\right)^{f}_{h} \left(F^{ahd}_{e}\right)^{g}_{k} \left(F^{bcd}_{k}\right)^{h}_{l} = \left(F^{fcd}_{e}\right)^{g}_{l} \left(F^{abl}_{e}\right)^{f}_{k}$ $\forall a,b,c,d,e,f,g,k\in \mathcal{I}$,
	\item [F.3] \textit{Left Inverse}: $\sum_{f\in\mathcal{I}} (F^{abc}_{d})^{e}_{f} (\bar{F}^{abc}_{d})^{f}_{e'} = \delta_{e}^{e'} (1-\delta_{N_{ab}^{e}}^{0})(1-\delta_{N_{ec}^{d}}^{0})$ $\forall a,b,c,d,e,e'\in \mathcal{I}$,
	\item [F.4] \textit{Right Inverse}: $\sum_{e\in\mathcal{I}} (\bar{F}^{abc}_{d})^{f'}_{e} (F^{abc}_{d})^{e}_{f} = \delta_{f}^{f'} (1-\delta_{N_{bc}^{f}}^{0})(1-\delta_{N_{af}^{d}}^{0})$ $\forall a,b,c,d,f,f'\in \mathcal{I}$.
\end{enumerate}
\end{dfn}
\noindent
Using the associativity axiom we are to extend the definition of the structure constants $N$
\begin{align*}
  N_{a_{0}\cdots a_{n}}^{b_{n}} & = \sum_{b_{1},\dots,b_{n-1}\in\mathcal{I}} N_{a_{0}a_{1}}^{b_{1}} \cdots N_{a_{n-1}a_{n}}^{b_{n}}. 
\end{align*}
We remark that our definition of a semi-simple tensor system presented here differs from other author's, specifically we have not required the existence of an identity or duals \cite{DaHaWa2013}. A semi-simple tensor system with such objects is defined in the follow way.
\begin{dfn}
A \textit{semi-simple tensor system} with identity and duals is a semi-simple tensor system that has a distinguished element, $1\in\mathcal{I}$, and an involution, $\iota$, satisfying:
\begin{enumerate}
	\item [N.3] \textit{Identity}: $N_{a1}^{b}=N_{1a}^{b}=\delta_{b}^{a}$ $\forall a,b\in\mathcal{I}$,
	\item [N.4] \textit{Duals}: $N_{ab}^{1}=N_{ba}^{1}=\delta_{b}^{\iota(a)}$ $\forall a,b\in\mathcal{I}$.
\end{enumerate}
\end{dfn}
\noindent
To connect tensor systems to categories we first must recognise there are different definitions of tensor (and related) categories, for example see \cite{DaHaWa2013,EGNO2009}. In this article tensor category means $\textbf{R}$-linear abelian monoidal category. It follows that given a semi-simple tensor system with identity and duals it is always possible to construct a semi-simple rigid tensor category with unit object \cite{Yamaga2002}. 
The category will consist of a set of simple objects $\mathcal{I}$ with the fusion of simple objects $a$ and $b$ given by
\begin{align*}
  a \tp b & \cong  \bigoplus_{c} N_{ab}^{c} c.
\end{align*}
Axiom N.1 implies that tensoring simple objects is associative, i.e.
\begin{align*}
  (a \tp b) \tp c \cong & a \tp (b \tp c),
\end{align*}
which is governed by the F-moves (alternatively known as generalised 6-j symbols). Associativity is naturally extended to all objects within the category.

As we are only dealing with multiplicity free systems we note
\begin{align*}
	(1-\delta_{N_{ab}^{c}}^{0}) & = N_{ab}^{c}, \\
	N_{ab}^{e}N_{bc}^{f}N_{af}^{d}N_{ef}^{d} \left(F^{abc}_{d}\right)^{e}_{f} & = \left(F^{abc}_{d}\right)^{e}_{f}, \\
	N_{ab}^{e}N_{bc}^{f}N_{af}^{d}N_{ef}^{d} \left(\bar{F}^{abc}_{d}\right)^{f}_{e} & = \left(\bar{F}^{abc}_{d}\right)^{f}_{e},
\end{align*}
$\forall a,b,c,d,e,f\in \mathcal{I}$. We also define equivalences between different semi-simple tensor systems. 
\begin{dfn}
Two semi-simple tensor systems, $(\mathcal{I},N,F,\bar{F})$ and $(\tilde{\mathcal{I}},\tilde{N},\tilde{F},\bar{\tilde{F}})$, are \textit{gauge equivalent} if $\mathcal{I}=\bar{\mathcal{I}}$,
\begin{align*}
	N_{ab}^{c} & = \tilde{N}_{ab}^{c} & \forall a,b,c\in\mathcal{I},
\end{align*}
and there exist \textit{gauge constants} $u^{ab}_{c}\in \textbf{R}^{\times}$ for $a,b,c\in\mathcal{I}$ such that
\begin{align*}
	\left(F^{abc}_{d}\right)^{e}_{f} & = \frac{u^{af}_{d}u^{bc}_{f}}{u^{ab}_{e}u^{ec}_{d}} \left(\tilde{F}^{abc}_{d}\right)^{e}_{f}, & \forall a,b,c,d,e,f\in\mathcal{I}\\
	\left(\bar{F}^{abc}_{d}\right)^{f}_{e} & = \frac{u^{ab}_{e}u^{ec}_{d}}{u^{af}_{d}u^{bc}_{f}} \left(\bar{\tilde{F}}^{abc}_{d}\right)^{f}_{e} & \forall a,b,c,d,e,f\in\mathcal{I}.
\end{align*}
\end{dfn}
\noindent
A quantity, associated with a semi-simple tensor system, is said to be gauge invariant if it is equal in all gauge equivalent semi-simple tensor systems, for example we find that
\begin{align*}
	\left(F^{aaa}_{d}\right)^{e}_{e},
\end{align*}
is gauge invariant.


We also define the braided version of semi-simple tensor system.
\begin{dfn}
A \textit{braided semi-simple tensor system} is semi-simple tensor system that has also has the property that $N_{ab}^{c} = N_{ba}^{c}$ $\forall a,b,c \in \mathcal{I}$ and there exists constants $R^{ab}_{c},\bar{R}^{ab}_{c}\in\textbf{R}$ satisfying
\begin{enumerate}
	\item [R.1] \textit{Zero Condition}: $N_{ab}^{c}=0$ implies $R^{ab}_{c}=\bar{R}^{ab}_{c}=0$ $\forall a,b,c\in\mathcal{I}$,
	\item [R.2] \textit{Hexagon Relation one}: $R^{ac}_{e}\left(F^{acb}_{d}\right)^{e}_{g} R^{bc}_{g}
=\sum_{f\in\mathcal{I}} \left(F^{cab}_{d}\right)^{e}_{f} R^{fc}_{d} \left(F^{abc}_{d}\right)^{f}_{g}$ $\forall a,b,c,d,e,g\in \mathcal{I}$,
	\item [R.3] \textit{Hexagon Relation two}: $\bar{R}^{ac}_{e}\left(F^{acb}_{d}\right)^{e}_{g} \bar{R}^{bc}_{g}
=\sum_{f\in\mathcal{I}} \left(F^{cab}_{d}\right)^{e}_{f} \bar{R}^{fc}_{d} \left(F^{abc}_{d}\right)^{f}_{g}$ $\forall a,b,c,d,e,g\in \mathcal{I}$,
	\item [R.4] \textit{Inverse}: $R^{ab}_{c}\bar{R}^{ba}_{c} = (1-\delta_{N_{ab}^{c}}^{0})$ $\forall a,b,c\in\mathcal{I}$.
\end{enumerate}
\end{dfn}
\noindent
The definition of gauge equivalent can be extended to the braided versions.
\begin{dfn}
Two braided semi-simple tensor systems, $(\mathcal{I},N,F,\bar{F},R,\bar{R})$ and $(\bar{\mathcal{I}},\tilde{N},\tilde{F},\bar{\tilde{F}},\tilde{R},\bar{\tilde{R}})$, are \textit{gauge equivalent} if they are gauge equivalent as semi-simple tensor systems with the gauge constants $u^{ab}_{c}\in \textbf{R}^{\times}$ for $a,b,c\in\mathcal{I}$ that also satisfy
\begin{align*}
	R^{ab}_{c} & = \frac{u^{ba}_{c}}{u^{ab}_{c}} \tilde{R}^{ab}_{c}, \\
	\bar{R}^{ab}_{c} & = \frac{u^{ab}_{c}}{u^{ba}_{c}}  \bar{\tilde{R}}^{ab}_{c}. 
\end{align*}
\end{dfn}
\noindent
The definition of gauge invariant is extended to the braided semi-simple tensor system.

The Temperley-Lieb and BMW like relations we will derive will be between operators that act on a Hilbert space, $\mathcal{H}_{\tilde{\lambda}}$, where $\tilde{\lambda}=(\lambda_{1},...,\lambda_{L})$ is an ordered $L$-tuple with $\lambda_{i}\in\mathcal{I}$. The Hilbert space is defined in the following way
\begin{align*}
	\mathcal{B}_{\tilde{\lambda}} & = \{\mu_{0}\mu_{1}...\mu_{ L} | N_{\mu_{i-1}\lambda_{i}}^{\mu_{i}}=1, \, 1 \leq i \leq L, \, \& \, \mu_{0}\in \mathcal{I}_{0}\}, \\
	\mathcal{H}_{\tilde{\lambda}} & = \textbf{R} \left\{ \ket{\mu} | \mu \in \mathcal{B}_{\tilde{\lambda}} \right\},
\end{align*}
for a fixed finite subset $\mathcal{I}_{0} \subseteq \mathcal{I}$. While $\mathcal{B}_{\tilde{\lambda}}$ and vector space $\mathcal{H}_{\tilde{\lambda}}$ both depend on $\mathcal{I}_{0}$, our results do not. It is because of this we suppress any notational indication of the dependence on $\mathcal{I}_{0}$. We have adopted the short hand that $\ket{\mu} \equiv \ket{\mu_{0}\mu_{1}...\mu_{ L}}$. As we require $\mathcal{I}_{0}$ to be finite axiom N.1 ensures that $\mathcal{H}_{\tilde{\lambda}}$ is finite-dimensional. The basis vectors of the dual space of $\mathcal{H}_{\tilde{\lambda}}$ are given by $\bra{\mu} \equiv \bra{\mu_{0}\mu_{1}...\mu_{ L}}$ and are orthogonal to the original basis vectors,
\begin{align*}
	\bra{\mu'} \ket{\mu} & = \prod_{i=0}^{L} \delta_{\mu_{i}}^{\mu_{i}'}.
\end{align*}
We follow the usual convection when defining an outer product. This Hilbert space and its basis is the same used for interaction-round-a-face models \cite{AnBaFo1984,Pasqui1988}.

We define two-site projection operators, 
\begin{align} \label{eqndfnProjOp}
	\bra{\mu'} p^{(\nu)}_{i} \ket{\mu} 
	& = \left[ \prod_{j\neq i} \delta_{\mu_{j}}^{\mu_{j}'} \right] \left(\bar{F}^{\mu_{i-1}\lambda_{i}\lambda_{i+1}}_{\mu_{i+1}}\right)_{\mu_{i}'}^{\nu} \left(F^{\mu_{i-1}\lambda_{i}\lambda_{i+1}}_{\mu_{i+1}}\right)^{\mu_{i}}_{\nu}  & 1\leq i\leq L-1.
\end{align}
It follows from the axioms of the semi-simple tensor systems that they are local, orthogonal and idempotent:
\begin{align*}
	p^{(\nu)}_{i} p^{(\nu')}_{j} & = p^{(\nu')}_{j} p^{(\nu)}_{i} & 1\leq i \leq j+2 \leq L-1, \\
	p^{(\nu)}_{i} p^{(\nu')}_{i} & = \delta_{\nu}^{\nu'} p^{(\nu)}_{i} & 1\leq i\leq L-1, \\
	\sum_{\nu\in\mathcal{I}}p^{(\nu)}_{i} & = I & 1\leq i\leq L-1.
\end{align*}

\section{Determining relations}
\subsection{Temperley-Lieb like relations}

\noindent
We start by show that projection operators satisfy Temperley-Lieb like relations if and only if there exists a relatively simple relationship between a gauge invariant quantities as stated in Result \ref{ResGenTL}.
\begin{thm}\label{ThmGenTLlikeRel}
Consider a semi-simple tensor system with indexing set $\mathcal{I}$ along with the elements $\nu,\nu'\in\mathcal{I}$ and the tuple $\tilde{\lambda}=(\lambda_{1},...,\lambda_{ L})$ where $\lambda_{i}\in\mathcal{I}$.

There exists $c_{1}\in\textbf{R}$ with the property that
\begin{align}
	c_{1} N_{\lambda_{i}\lambda_{i+1}}^{\nu} N_{\nu\lambda_{i+2}}^{\upsilon_{1}} N_{\upsilon_{2}\upsilon_{1}}^{\upsilon_{3}}
	& = N_{\upsilon_{2}\upsilon_{1}}^{\upsilon_{3}} \left(F^{\lambda_{i}\lambda_{i+1}\lambda_{i+2}}_{\upsilon_{1}}\right)^{\nu}_{\nu'} 	\left(\bar{F}^{\lambda_{i}\lambda_{i+1}\lambda_{i+2}}_{\upsilon_{1}}\right)^{\nu'}_{\nu}, \label{eqncFFone}
\end{align}
for all $\upsilon_{1},\upsilon_{2},\upsilon_{3}\in\mathcal{I}$ where $\sum_{\mu_{0}\in\mathcal{I}_{0}} N_{u_{0}\lambda_{1}\cdots\lambda_{i-1}}^{\upsilon_{2}} \geq 1$ if and only if there exists $c_{2}\in\textbf{R}$ with the property that
\begin{align*}
	p_{i}^{(\nu)}p_{i+1}^{(\nu')}p_{i}^{(\nu)} & = c_{2} p_{i}^{(\nu)}
\end{align*}
as operators acting on the space $\mathcal{H}_{\tilde{\lambda}}$. If such $c_{1}$ and $c_{2}$ exist it is always possible to set $c_{1}=c_{2}$. 

There exists $c_{1}\in\textbf{R}$ with the property that
\begin{align}
	c_{1} N_{\lambda_{i}\lambda_{i+1}}^{\nu} N_{\lambda_{i-1}\nu}^{\upsilon_{1}} N_{\upsilon_{2}\upsilon_{1}}^{\upsilon_{3}}  
	& = N_{\upsilon_{2}\upsilon_{1}}^{\upsilon_{3}} \left(\bar{F}^{\lambda_{i-1}\lambda_{i}\lambda_{i+1}}_{\upsilon_{1}}\right)^{\nu}_{\nu'} 		\left(F^{\lambda_{i-1}\lambda_{i}\lambda_{i+1}}_{\upsilon_{1}}\right)^{\nu'}_{\nu} 	\label{eqncFFtwo}
\end{align}
for all $\upsilon_{1},\upsilon_{2},\upsilon_{3}\in\mathcal{I}$ where $\sum_{\mu_{0}\in\mathcal{I}_{0}} N_{u_{0}\lambda_{1}\cdots\lambda_{i-2}}^{\upsilon_{2}} \geq 1$ if and only if there exists $c_{2}\in\textbf{R}$ with the property that
\begin{align*}
	p_{i}^{(\nu)}p_{i-1}^{(\nu')}p_{i}^{(\nu)} & = c_{2} p_{i}^{(\nu)}
\end{align*}
as operators acting on the space $\mathcal{H}_{\tilde{\lambda}}$. If such $c_{1}$ and $c_{2}$ exist it is always possible to set $c_{1}=c_{2}$.
\end{thm}
\begin{proof}
The proof of this proposition can be found in Section \ref{SecProofGenTLlikeRel} (Calculations \ref{calcTLlikeRel1} and \ref{calcTLlikeRel2}).
\end{proof}

\noindent
This gives sufficient and necessary condition for the projection operator to satisfy Temperley-Lieb like relations corresponding to a constraint on a set of gauge invariant quantities. This is can be rewritten to coincide with Result 1 presented in Outline section.

Unfortunately the above Theorem does not provide much insight into cases where one find projection operators satisfying Temperley-Lieb like relations. It is useful to consider a sufficient constraint purely on the coefficients $N_{ab}^{c}$. To accomplish this we require the definition of being acting one-dimensionally. 
\begin{dfn}
An element $\nu\in \mathcal{I}$ is said act \textit{one-dimensionally on the left of} $\mu\in \mathcal{I}$ if
\begin{align*}
	\sum_{\mu'\in I} N_{\nu\mu}^{\mu'} & = 1.
\end{align*}
Similarly, $\nu\in \mathcal{I}$ is said acts \textit{one-dimensionally on the right of} $\mu\in \mathcal{I}$ if
\begin{align*}
	\sum_{\mu'\in I} N_{\mu\nu}^{\mu'} & = 1.
\end{align*}
Finally, if $\nu\in \mathcal{I}$ acts one-dimensionally on the left and right of an element then it is said that $\nu$ acts one-dimensionally on that element.
\end{dfn}

\noindent
In terms of objects in the category we see that $a\tp b$ is simple if and only if $a$ acts one-dimensionally on the right of $b$, which is equivalent to stating $b$ acts one-dimensionally on the left of $a$.

\begin{dfn}
For $\nu\in\mathcal{I}$ we define the sets
\begin{align*}
	\mathcal{I}^{l}_{\nu} & = \{\mu\in\mathcal{I}| \,\nu\,\, \mbox{acts one-dimensionally on the left of} \,\,\mu\, \}, \\
	\mathcal{I}^{r}_{\nu} & = \{\mu\in\mathcal{I}| \,\nu\,\, \mbox{acts one-dimensionally on the right of} \,\,\mu\, \},
\end{align*}
along with the maps
\begin{align*}
	\phi^{l}_{\nu} : \mathcal{I}^{l}_{\nu} \rightarrow \mathcal{I} & \hspace{0.4cm} \mbox{s.t.} \hspace{0.4cm} N_{\nu\mu}^{\mu'} = \delta_{\phi^{l}_{\nu}(\mu)}^{\mu'}, \\
	\phi^{r}_{\nu} : \mathcal{I}^{r}_{\nu} \rightarrow \mathcal{I} & \hspace{0.4cm} \mbox{s.t.} \hspace{0.4cm} N_{\mu\nu}^{\mu'} = \delta_{\phi^{r}_{\nu}(\mu)}^{\mu'}.
\end{align*}
\end{dfn}

\begin{dfn}
An element $\nu\in \mathcal{I}$ is said to be \textit{one-dimensional} if $\mathcal{I}=\mathcal{I}^{l}_{\nu}=\mathcal{I}^{r}_{\nu}$.
\end{dfn}
\noindent
The reason the element is called one-dimensional is because in a semi-simple tensor system with identity and duals that also has a finite number of simple objects the (non-trivial) quantum dimension of the corresponding object in the must have magnitude one. Furthermore, if it is also a category of irreducible finite-dimensional representations of a semi-simple algebra then the dimension of the corresponding representation will be precisely one. It should be noted that in conformal field theory such objects are called simple currents.

The maps $\phi^{l}_{\nu}$ and $\phi^{r}_{\nu}$ allows us to perform some interesting calculations and find the relationships between projection operators involving elements that act one-dimensionally.

\begin{cor} \label{CorOneDProjOpRel}
Consider a semi-simple tensor system with indexing set $\mathcal{I}$ along with the tuple $\tilde{\lambda}=(\lambda_{1},...,\lambda_{ L})$ where $\nu,\lambda_{i}\in\mathcal{I}$. For a fixed $1 \leq i \leq L-2$, if $\nu$ acts \underline{one-dimensionally} on the \underline{left} of $\lambda_{i+2}$ then as operators acting on the Hilbert space $\mathcal{H}_{\tilde{\lambda}}$,
\begin{align*}
	p^{(\nu)}_{i}p^{(\nu')}_{i+1}p^{(\nu)}_{i} 
	& = \left(F^{\lambda_{i}\lambda_{i+1}\lambda_{i+2}}_{\phi^{l}_{\nu}(\lambda_{i+2})}\right)^{\nu}_{\nu'} \left(\bar{F}^{\lambda_{i}\lambda_{i+1}\lambda_{i+2}}_{\phi^{l}_{\nu}(\lambda_{i+2})}\right)^{\nu'}_{\nu} p^{(\nu)}_{i}, 
	& \forall v'\in\mathcal{I}.
\end{align*}
Similarly, for a fixed $2 \leq i \leq L-1$, if $\nu$ acts \underline{one-dimensionally} on the \underline{right} of $\lambda_{i-1}$ then as operators acting on the Hilbert space $\mathcal{H}_{\tilde{\lambda}}$,
\begin{align*}
	p^{(\nu)}_{i}p^{(\nu')}_{i-1}p^{(\nu)}_{i} 
	& = \left(F^{\lambda_{i-1}\lambda_{i}\lambda_{i+1}}_{\phi^{r}_{\nu}(\lambda_{i-1})}\right)^{\nu'}_{\nu} \left(\bar{F}^{\lambda_{i-1}\lambda_{i}\lambda_{i+1}}_{\phi^{r}_{\nu}(\lambda_{i-1})}\right)^{\nu}_{\nu'} p^{(\nu)}_{i}, 
	& \forall v'\in\mathcal{I}.
\end{align*}
\end{cor}
\begin{proof}
This is a corollary of Theorem \ref{ThmGenTLlikeRel}.

\end{proof}

\noindent
The above corollary corresponds to Result \ref{ResPart1DTL} recalling that when one simple object acts one-dimensionally on another it is the same as saying their tensor product is also simple. We remark that if $\nu$ in the above corollary is one-dimensional then conditions will automatically be met. This was described in Result \ref{Res1DTL}.

It is important to recall that while Theorem \ref{ThmGenTLlikeRel} provides sufficient and necessary conditions and Corollary \ref{CorOneDProjOpRel} provides only sufficient conditions the author know of no non-trivial case where aforementioned theorem can be applied but the corollary can not. The existence of such an example was presented in Open Problem \ref{op12} and here we present a more precise alternate version.

\begin{center} \begin{minipage}{0.9\textwidth}
\vspace{0.2cm}
\begin{opAlt1}
Does there exist a semi-simple tensor system with indexing set $\mathcal{I}$ and $\nu,\nu',\lambda_{1},\lambda_{2},\lambda_{3} \in\mathcal{I}$ such that either
\begin{align*}
	c N_{\lambda_{1}\lambda_{2}}^{\nu} N_{\nu\lambda_{3}}^{\upsilon_{1}} N_{\upsilon_{2}\upsilon_{1}}^{\upsilon_{3}}
	& = N_{\upsilon_{2}\upsilon_{1}}^{\upsilon_{3}} \left(F^{\lambda_{1}\lambda_{2}\lambda_{3}}_{\upsilon_{1}}\right)^{\nu}_{\nu'} 	\left(\bar{F}^{\lambda_{1}\lambda_{2}\lambda_{3}}_{\upsilon_{1}}\right)^{\nu'}_{\nu} ,
	& \upsilon_{1},\upsilon_{2},\upsilon_{3}\in\mathcal{I}, \quad \mbox{and} \\
	\sum_{\mu} N_{\nu\lambda_{3}}^{\mu} & \quad \geq \quad 1,
\end{align*}
for a unique non-zero $c\in \textbf{R}$ or
\begin{align*}
	c N_{\lambda_{2}\lambda_{3}}^{\nu} N_{\lambda_{1}\nu}^{\upsilon_{1}} N_{\upsilon_{2}\upsilon_{1}}^{\upsilon_{3}}  
	& = N_{\upsilon_{2}\upsilon_{1}}^{\upsilon_{3}} \left(\bar{F}^{\lambda_{1}\lambda_{2}\lambda_{3}}_{\upsilon_{1}}\right)^{\nu}_{\nu'} 		\left(F^{\lambda_{1}\lambda_{2}\lambda_{3}}_{\upsilon_{1}}\right)^{\nu'}_{\nu} ,
	& \upsilon_{1},\upsilon_{2},\upsilon_{3}\in\mathcal{I}, \quad \mbox{and}\\
	\sum_{\mu} N_{\lambda_{1}\nu}^{\mu} & \quad \geq \quad 1,
\end{align*}
for a unique non-zero $c\in \textbf{R}$?
\end{opAlt1}
\vspace{0.2cm}
\end{minipage} \end{center}

\noindent
As we wish to avoid uninteresting cases, such as $N_{\lambda_{1}\lambda_{2}}^{\nu}=0$, we also included the requirement of a unique non-zero $c$. 

We can also ask if there is a difference between an object acting one-dimensionally and being one-dimensional. This can be phased as does $a\tp b$ being simple imply that $a \tp c$ is simple for all simple objects $c$ or $c \tp b$ is simple for all simple objects $c$? 

The answer is no. This is seen by taking the Fibonacci anyons (Fib) with objects $\mathcal{I}=\{1,\tau\}$ where $1$ is the identity (satisfies N.3) and has the relation $\tau \tp \tau \cong 1 \oplus \tau$. One can then take Fib$\times$Fib which has objects $\mathcal{I}=\{1,1),(1,\tau),(\tau,1),(\tau,\tau)\}$ and has, amongst others, the fusion relations
\begin{align*}
 (1,\tau) \tp (\tau,1) & \cong (\tau,\tau), \\
 (1,\tau) \tp (\tau,\tau) & \cong (\tau,\tau) \oplus (\tau,1), \\
 (\tau,\tau) \tp (\tau,1) & \cong (\tau,\tau) \oplus (1,\tau).
\end{align*}
While this example answers the question it is deceptive as it still relies on a one-dimensional object, namely $1$. 
Consequently, we wish to ignore cases where the semi-simple tensor category is the direct product (as defined in \cite{BonderThesis2007}) of other semi-simple tensor categories, thus we arrive at the Open Problem \ref{op23}.


There is a further special case which uncovers the usual Temperley-Lieb relations.
\begin{cor} \label{corUsualTL}
Consider a semi-simple tensor system with objects indexed by $\mathcal{I}$ and two element $\lambda,\nu\in\mathcal{I}$ such that $\nu$ acts one-dimensionally on $\lambda$ and $N_{\lambda\lambda}^{\nu}=1$. If there exists $d\in\textbf{R}^{\times}$ such that
\begin{align*}
  d^{-2} & = \left(F^{\lambda\lambda\lambda}_{\phi_{\nu}(\lambda)}\right)^{\nu}_{\nu} \left(\bar{F}^{\lambda\lambda\lambda}_{\phi_{\nu}(\lambda)}\right)^{\nu}_{\nu},
\end{align*}
then acting on the Hilbert space $\mathcal{H}_{\tilde{\lambda}}$ where $\tilde{\lambda}=(\lambda,..,\lambda)$  the operators
\begin{align} \label{eqnUopHomo}
	U_{i} & = d p_{i}^{(\nu)} & 1 \leq i \leq L-1,
\end{align}
satisfy the Temperley-Lieb relations
\begin{align*}
	U_{i}^{2} & = d U_{i} & 1\leq i\leq L-1, \\
	U_{i}U_{j} & = U_{j}U_{i} & 1\leq i \leq j+2 \leq L-1, \\
	U_{i}U_{i\pm1}U_{i} & = U_{i} & 1\leq i,i\pm1 \leq L-1.
\end{align*}
\end{cor}
\begin{proof}
This follows directly from the previous corollary.
\end{proof}

\subsection{BMW like relations}
This section deals with the construction of BMW like relations and requires braided fusion categories and the existence of elements that act one-dimensionally. We first recall that for braided semi-simple tensor systems $N_{ab}^{c}=N_{ba}^{c}$ which implies that for all $\nu\in\mathcal{I}$ that $\mathcal{I}^{l}_{\nu}=\mathcal{I}^{r}_{\nu}$ and $\phi^{l}_{\nu} = \phi^{r}_{\nu}$. For convenience we define $\phi_{\nu}=\phi^{l}_{\nu}$. \\

\noindent
For any Hilbert space $\mathcal{H}_{\tilde{\lambda}}$ with $\tilde{\lambda}=(\lambda,...,\lambda)$ we have the operator
\begin{align} \label{eqnRmathomo}
	R_{i} & = \sum_{\mu} R^{\lambda\lambda}_{\nu} p_{i}^{(\mu)}.
\end{align}
It follows from the definition of the $R^{ab}_{c}$ that $R_{i}$ is invertible. Moreover, we have (I haven't proved this yet but am sure it is known)
\begin{align}
	R_{i}R_{i+1}R_{i} & = R_{i+1}R_{i}R_{i+1} 		& 1 \leq i \leq L-2, \label{eqnBraid} \\
	R_{i}R_{i'} & =  R_{i'}R_{i}	 								& 1 \leq i \leq i'-1 \leq L-2. \label{eqnBraidCom}
\end{align}
We are able to determine the following BMW like relations.
\begin{lem} \label{lemBraidaTLop}
Consider a semi-simple tensor system with objects indexed by $\mathcal{I}$ and two element $\lambda,\nu\in\mathcal{I}$ such that $\nu$ acts one-dimensionally on $\lambda$. If there exists $c\in\R^{\times}$ such that
\begin{align*}
  c & = \left(F^{\lambda\lambda\lambda}_{\phi_{\nu}(\lambda)}\right)^{\nu}_{\nu} \left(\bar{F}^{\lambda\lambda\lambda}_{\phi_{\nu}(\lambda)}\right)^{\nu}_{\nu},
\end{align*}
then the operators $p_{i}^{(\nu)}$ and $R_{i}$ defined in Equations (\ref{eqndfnProjOp}) and (\ref{eqnRmathomo}) acting on the Hilbert space $\mathcal{H}_{\tilde{\lambda}}$ with $\tilde{\lambda}=(\lambda,...,\lambda)$, satisfy the relations
\begin{align*}
	R_{i}p_{i}^{(\nu)} & = p_{i}^{(\nu)}R_{i} = R^{\lambda\lambda}_{\nu} p_{i}^{(\nu)}, \\
	R_{i}^{-1}p_{i}^{(\nu)} & = p_{i}^{(\nu)}R_{i}^{-1} = \bar{R}^{\lambda\lambda}_{\nu} p_{i}^{(\nu)}, \\
	p_{i}^{(\nu)}R_{i\pm1}p_{i}^{(\nu)} & = \left[\sum_{\upsilon\in\mathcal{I}} \left(F^{\lambda\lambda\lambda}_{\phi_{\nu}(\lambda)}\right)^{\upsilon}_{\nu}  R^{\lambda\lambda}_{\upsilon} \left(\bar{F}^{\lambda\lambda\lambda}_{\phi_{\nu}(\lambda)}\right)^{\nu}_{\upsilon} \right] p_{i}^{\nu}, \\
	p_{i}^{(\nu)}R_{i\pm1}^{-1}p_{i}^{(\nu)} & = \left[\sum_{\upsilon\in\mathcal{I}}  \left(F^{\lambda\lambda\lambda}_{\phi_{\nu}(\lambda)}\right)^{\upsilon}_{\nu} \bar{R}^{\lambda\lambda}_{\upsilon}  \left(\bar{F}^{\lambda\lambda\lambda}_{\phi_{\nu}(\lambda)}\right)^{\nu}_{\upsilon} \right] p_{i}^{\nu}, \\
	R_{i}R_{i\pm1}p_{i}^{(\nu)} 
	& = p_{i\pm1}^{(\nu)}R_{i}R_{i\pm1} 
	= c^{-1} \left[R^{\lambda\lambda}_{\nu} \sum_{\upsilon\in\mathcal{I}} \left(F^{\lambda\lambda\lambda}_{\phi_{\nu}(\lambda)}\right)^{\upsilon}_{\nu}  R^{\lambda\lambda}_{\upsilon}  \left(\bar{F}^{\lambda\lambda\lambda}_{\phi_{\nu}(\lambda)}\right)^{\nu}_{\upsilon} \right] p_{i\pm1}^{(\nu)}p_{i}^{(\nu)} \\
	R_{i}^{-1}R_{i\pm1}^{-1}p_{i}^{(\nu)},
	& = p_{i\pm1}^{(\nu)}R_{i}^{-1}R_{i\pm1}^{-1} 
	= c^{-1} \left[\bar{R}^{\lambda\lambda}_{\nu} \sum_{\upsilon\in\mathcal{I}} \left(F^{\lambda\lambda\lambda}_{\phi_{\nu}(\lambda)}\right)^{\upsilon}_{\nu}  R^{\lambda\lambda}_{\upsilon}  \left(\bar{F}^{\lambda\lambda\lambda}_{\phi_{\nu}(\lambda)}\right)^{\nu}_{\upsilon} \right] p_{i\pm1}^{(\nu)}p_{i}^{(\nu)}.
\end{align*}
\end{lem}
\begin{proof}
The first and second relations are consequence of $p_{i}^{(\nu)}$ being an idempotent projection operator. The third and fourth relations follow from Corollary \ref{CorOneDProjOpRel}. The fifth and sixth relations are shown in Calculations \ref{calcRRppp1}-\ref{calcRRppp3}.
\end{proof}
\noindent
It should be noted that using the hexagon relation we are able to write the expression(s)
\begin{align*}
	 \sum_{\upsilon\in\mathcal{I}} \left(F^{\lambda\lambda\lambda}_{\phi_{\nu}(\lambda)}\right)^{\upsilon}_{\nu}  R^{\lambda\lambda}_{\upsilon}  \left(\bar{F}^{\lambda\lambda\lambda}_{\phi_{\nu}(\lambda)}\right)^{\nu}_{\upsilon} 
	 & = R^{\nu\lambda}_{\phi_{\nu}(\lambda)} \left(F^{\lambda\lambda\lambda}_{\phi_{\nu}(\lambda)}\right)^{\nu}_{\nu} \bar{R}^{\lambda\lambda}_{\nu} \\
	 & = R^{\lambda\lambda}_{\nu} \left(\bar{F}^{\lambda\lambda\lambda}_{\phi_{\nu}(\lambda)}\right)^{\nu}_{\nu} \bar{R}^{\lambda\nu}_{\phi_{\nu}(\lambda)} \\
	 & = \bar{R}^{\nu\lambda}_{\phi_{\nu}(\lambda)} \left(F^{\lambda\lambda\lambda}_{\phi_{\nu}(\lambda)}\right)^{\nu}_{\nu} R^{\lambda\lambda}_{\nu} \\
	 & = \bar{R}^{\lambda\lambda}_{\nu} \left(\bar{F}^{\lambda\lambda\lambda}_{\phi_{\nu}(\lambda)}\right)^{\nu}_{\nu} R^{\lambda\nu}_{\phi_{\nu}(\lambda)}
\end{align*}

\begin{thm} \label{ThmBMW}
Consider a semi-simple tensor system with objects indexed by $\mathcal{I}$ and two element $\lambda,\nu\in\mathcal{I}$ such that $\nu$ acts one-dimensionally on $\lambda$. If there exists $d,g\in\R^{\times}$ such that
\begin{align*}
  d^{-1} & = \sum_{\upsilon\in\mathcal{I}} \left(F^{\lambda\lambda\lambda}_{\phi_{\nu}(\lambda)}\right)^{\upsilon}_{\nu}  R^{\lambda\lambda}_{\upsilon}  \left(\bar{F}^{\lambda\lambda\lambda}_{\phi_{\nu}(\lambda)}\right)^{\nu}_{\upsilon},\\
  g^{-2} & = R^{\lambda\lambda}_{\nu}
\end{align*}
then the operators 
\begin{align*}
	U_{i} & = d p_{i}^{(\nu)} & 1 \leq i \leq L-1, \\
	G_{i} & = g R_{i} & 1 \leq i \leq L-1,
\end{align*}
acting on the Hilbert space $\mathcal{H}_{\tilde{\lambda}}$ with $\tilde{\lambda}=(\lambda,...,\lambda)$, satisfy the relations
\begin{align*}
	G_{i}G_{i+1}G_{i} & = G_{i+1}G_{i}G_{i+1} 		& 1 \leq i \leq L-2, \\
	U_{i}U_{i\pm1}U_{i} & =  U_{i} 								& 1 \leq i,(i\pm1) \leq L-1, \\
	G_{i}G_{i\pm1}U_{i} & = U_{i\pm1}G_{i}G_{i\pm1} = U_{i\pm1}U_{i} 				& 1 \leq i,(i\pm1) \leq L-1, \\
	G_{i\pm1}U_{i}G_{i\pm1} & = G_{i}^{-1}U_{i\pm1}G_{i}^{-1} 									& 1 \leq i,(i\pm1) \leq L-1, \\
	U_{i}G_{i\pm1}U_{i} & = g U_{i} 							& 1 \leq i,(i\pm1) \leq L-1, \\
	G_{i} U_{i} & = U_{i} G_{i} = g^{-1} U_{i} 															& 1 \leq i \leq L-1, \\
	U_{i}^{2} & = d U_{i} 												& 1 \leq i \leq L-1, \\
	G_{i}G_{i'} & =  G_{i'}G_{i} 									& 1 \leq i \leq i'-1 \leq L-2, \\
	U_{i}U_{i'} & =  U_{i'}U_{i} 									& 1 \leq i \leq i'-1 \leq L-2. 
\end{align*}
\end{thm}
\begin{proof}
We start by first noting that
\begin{align*}
 d^{-2} 
 & = \left[\sum_{\upsilon\in\mathcal{I}} \left(F^{\lambda\lambda\lambda}_{\phi_{\nu}(\lambda)}\right)^{\upsilon}_{\nu}  R^{\lambda\lambda}_{\upsilon}  \left(\bar{F}^{\lambda\lambda\lambda}_{\phi_{\nu}(\lambda)}\right)^{\nu}_{\upsilon}\right]^{2} 
 = \left(F^{\lambda\lambda\lambda}_{\phi_{\nu}(\lambda)}\right)^{\nu}_{\nu} \left(\bar{F}^{\lambda\lambda\lambda}_{\phi_{\nu}(\lambda)}\right)^{\nu}_{\nu}.
\end{align*}
Subsequently, all the relations follow from Lemma \ref{lemBraidaTLop}, Corollary \ref{corUsualTL}, and Equations (\ref{eqnBraid}) and (\ref{eqnBraidCom}).
\end{proof}

\noindent
The above theorem provides conditions that imply that all bar one of the relations required in the Birman-Murakami-Wenzl algebra. The remaining relation is
\begin{align} \label{eqnBMWEigRel}
	G_{i} - G_{i}^{-1} & = m (I - U_{i}) & 1 \leq i \leq L-1.
\end{align}
We note that all the operators in the above relation commute and hence can be simultaneously diagonalised. Furthermore, the $U_{i}$ only have two eigenvalues, namely, 0 and $d$. It follows that Equation (\ref{eqnBMWEigRel}) is a condition on the eigenvalues of the $G_{i}$, specifically that there are at most three and they are not independent. Expanding the operators $U_{i}$ and $G_{i}$ in terms of projection operators leads us to conclude that Equation (\ref{eqnBMWEigRel}) is satisfied if and only if
\begin{align*}
	gR^{\lambda\lambda}_{\upsilon} - (gR^{\lambda\lambda}_{\upsilon})^{-1} & = m & \forall \upsilon \, \mbox{s.t.}\, N_{\lambda\lambda}^{\upsilon}=1, \, \upsilon \neq \nu \\
	g^{-1}-g & = m(1-d),
\end{align*}
with $d$ and $g$ being the values given in Theorem \ref{ThmBMW}.

\section{Examples and Limitations}
In this section we provide examples of either fusion categories or representations of Hopf algebras in which the one can use the results of the previous section to show the existence of operators satisfying Temperley-Lieb like relations. For the examples presented only the fusion rules or the dimension of the representations will be given, this is to demonstrate simplicity of the application of the Corollary \ref{CorOneDProjOpRel} and \ref{corUsualTL}.

The examples will also connect the results of this paper with known cases mentioned previously as well as show the potential limitations of the results.

\subsection{$U_{q}(su(2))$ and $su(2)_{k}$}
The $su(2)_{k}$ fusion categories consists of the set of objects $\mathcal{I}=\{0,\frac{1}{2},\dots,\frac{k}{2}\}$ along with the fusion rules
\begin{align*}
  a \tp b & \cong \bigoplus_{c=|a-b|}^{\mbox{min}(a+b,k-a-b)} c
\end{align*}
where $c-a-b$ must always be an integer. It is know that there are compatible F-moves, see \cite{BonderThesis2007}. If one takes $k$ to $\infty$ the $su(2)_{k}$ fusion categories correspond to the category of irreducible spin representations of the Hopf algebra $U_{q}(su(2))$ with generic $q$. It is easy to verify that the $0\in\mathcal{I}$ is an identity object and every object is its own dual,
\begin{align*}
  0 \tp a & \cong a \cong a \tp 0, \\
  a \tp a & \cong 0 \oplus \cdots.
\end{align*}
Thus if one considers $a\tp a\tp \cdots \tp a$, it follows from Corollary \ref{corUsualTL} that appropriately scaled two-site projection operators from $a\tp a$ to $0$ must satisfy Temperley-Lieb relations. For $a=\frac{1}{2}$ and $k=\infty$, i.e. $U_{q}(su(2))$, this is precisely Schur-Weyl duality and for $a=\frac{1}{2}$ and $k<\infty$ this results is known from Pasquier's construction of integrable models from A-D-E diagrams \cite{Pasquier1987a}.

\subsection{Group Hopf Algebras and Their Drinfeld Doubles}
Give a group $G$, one can consider two Hopf algebras either $H=\C G$ or $H=D(\C G)$ (the Drinfeld double of the group, see \cite{Gould1993} for a definition). Typically $\C G$ is treated as a subalgebra of $D(\C G)$, however, here they are treated separately for those unfamiliar Drinfeld doubles can ignore that case.

Consider all the finite irreducible matrix representations $\{\pi_{a}\}_{a\in\mathcal{I}}$ of $H$ with $V_{a}$ being the vector associated with $\pi_{a}$ and $d_{a}$ being its dimension. It follows that $\{\pi_{j}\}_{j\in\mathcal{I}}$ gives rise to a semi-simple tensor category. Given the tuple $\tilde{\lambda}=(\lambda_{1},...,\lambda_{ L})$ one has the has two-site projection operators
\begin{align*}
  p_{i}^{(\nu)} & = I_{\lambda_{1}} \tp \cdots \tp I_{\lambda_{i-1}} \tp p \tp I_{\lambda_{i+2}} \tp \cdots \tp I_{\lambda_{L}}, & \mbox{where}\\
  p & = \frac{d_{\nu}}{|G|}  \begin{dcases}
	  \sum_{g\in G} \mbox{tr}[\pi_{\nu}(g^{-1})]\, \pi_{\lambda_{i}}(g) \tp \pi_{\lambda_{i+1}}(g), & H=\C G, \\
	  \sum_{g,h,k\in G} \mbox{tr}[\pi_{\nu}(h^{*}g^{-1})]\, \pi_{\lambda_{i}}(g(k^{-1}h)^{*} \tp \pi_{\lambda_{i+1}}(gk^{*}), & H=D(\C G),
	\end{dcases}
\end{align*}
$|G|$ is the order of the group and $I_{\lambda}$ is the identity operator on space $V_{\lambda}$. In the case the tensor category is multiplicity free and $d_{\nu}=1$ then we apply Corollary \ref{CorOneDProjOpRel} and obtain
\begin{align*}
  p_{i}^{(\nu)} p_{i\pm1}^{(\nu')} p_{i}^{(\nu)} & = p_{i}^{(\nu)}.
\end{align*}
This is an equation on the space $V_{\lambda_{1}} \tp \cdots \tp V_{\lambda_{L}}$, rather than $\mathcal{H}_{\tilde{\lambda}}$. The two spaces are related through a face-vertex correspondence.

Gould \cite{Gould1993} presented a similar result. The key differences being that Gould required $\pi_{\nu}=\pi_{\nu'}$ to be the trivial representation and $\tilde{\lambda}$ to be homogeneous and self-dual but did not require fusion category of irreducible representations to be multiplicity free. This gives a indication that it might be possible to extend the results to systems which are not multiplicity free.

\subsection{A Deceptive Example Involving an Identity}
In any semi-simple tensor category system with an identity and duals a projection operator onto a one-dimensional object can be mapped to a projection operator onto the identity object in an isomorphic space. Specifically, for a fixed $\tilde{\lambda}$, $i$ and one-dimensional object $\nu$ there is always an isomorphism $f$ and sequence $\tilde{\lambda}'$ with the property
\begin{align*}
  f:\mathcal{H}_{\tilde{\lambda}} \rightarrow \mathcal{H}_{\tilde{\lambda}'} \quad \mbox{such that} \quad f(p_{i}^{(\nu)}) = p_{i}^{(1)},
\end{align*}
where $1$ is the identity object. This is a consequence of the dual of $\nu$ also being one-dimensional. From this point of view there is no distinction between projection operators onto the identity or other one-dimensional objects. In fact, if for a $\tilde{\lambda}$ there exists a set of one-dimensional objects $\{\nu_{i}\}$ such that $p_{i}^{(\nu_{i})} \neq 0$, then there exists a map
\begin{align*}
  f:\mathcal{H}_{\tilde{\lambda}} \rightarrow \mathcal{H}_{(\lambda_{1},\bar{\lambda}_{1},\lambda_{1},\bar{\lambda}_{1},\dots)} \quad \mbox{such that} \quad f(p_{i}^{(\nu_{i})}) = p_{i}^{(1)},
\end{align*}
where $\bar{\lambda}_{1} = \iota(\lambda_{1})$ is the dual of $\lambda_{1}$. This can be proved by considering starting with a chain of length $L=2$ and using induction.

This leads us to search for tensor systems which lacks an identity and duals. A simple example tensor system is readily available:
\begin{table}[ht]
\begin{center}
\begin{tabular}{|c|c|c|c|c|c|c|} \hline
	$\tp$ & $0$ & $1$ & $\tau$ \\ \hline
	$0$ & $1$ & $1$ & $\tau$ \\ \hline
	$1$ & $1$ & $1$ & $\tau$ \\ \hline
	$\tau$ & $\tau$ & $\tau$ & $1\oplus\tau$ \\ \hline
\end{tabular}
\end{center}
\end{table}

\noindent
The $F$-moves can be calculated by solving the pentagon relations. This example has no identity and no dual objects yet still allows for a non-trivial application of Corollary \ref{CorOneDProjOpRel}, consider $\tilde{\lambda}=(\tau,\dots,\tau)$. However, it becomes apparent that the objects $\{1,\tau\}$ leads to a tensor system isomorphic to the integer subsector of $su(2)_{3}$, i.e. Fibonacci anyons, with $1$ being the identity object.

\subsection{$\C$ and Homogeneity of the Constants}
In the special case where $\textbf{R} = \C$ and we just consider relationships between projection operators onto one-dimensional objects then we find the constants appearing in the Temperley-Lieb like relations are constrained. Consider $\tilde{\lambda}$ and a set of one-dimensional objects $\{\nu_{i}\}$ such that $p_{i}^{(\nu_{i})} \neq 0$. It follows that
\begin{align*}
  p_{i} & = c_{i}^{-1} p_{i}p_{i+1}p_{i} = c_{i-1}^{-1} p_{i}p_{i-1}p_{i} \\
\end{align*}
where
\begin{align*}
  p_{i} & = p_{i}^{(\nu_{i})} \\
  c_{i} & = \left(F^{\lambda_{i}\lambda_{i+1}\lambda_{i+2}}_{\phi^{l}_{\nu_{i+1}}(\lambda_{i+2})}\right)^{\nu_{i}}_{\nu_{i+1}} \left(\bar{F}^{\lambda_{i}\lambda_{i+1}\lambda_{i+2}}_{\phi^{l}_{\nu_{i+1}}(\lambda_{i+2})}\right)^{\nu_{i+1}}_{\nu_{i}}.
\end{align*}
Using this we find that 
\begin{align*}
  0 & = p_{i}p_{i+2}p_{i+1}p_{i+2}p_{i} - p_{i+2}p_{i}p_{i+1}p_{i}p_{i+2} = c_{i+1}p_{i}p_{i+2}p_{i} - c_{i}p_{i+2}p_{i}p_{i+2}  = (c_{i+1} - c_{i}) p_{i}p_{i+2}, 
\end{align*}
implying that all the $c_{i}$ is in actually independent of $i$. It remains to be seen if there exist examples of tensor systems with some general ring $\textbf{R}$ that allow different constants $c_{i}$ between projection operators onto one-dimensional objects.

\section{Conclusion}
We have provided the necessary and sufficient conditions for two-site projection operators satisfy Temperley-Lieb like relations in a multiplicity semi-simple tensor systems. Using Yamagami's reconstruction theorem and Pasquier's face-vertex correspondence one recovers many known examples of representations of the Temperley-Lieb algebra appearing in integrable systems as special cases our theorem. Despite our broad and abstract approach there is still a significant opportunity to generalise the theorem further by removing the multiplicity free conditions and checking if similar necessary and sufficient conditions exist. Likewise, it may be possible to perform similar calculations to find projection operators satisfying the Hecke relations.

\subsection*{Acknowledgements}
This manuscript was completed over an excessively long period, during which I also was based at the University of Leeds and Universit\'e Pierre-et-Marie-Curie. I would also like to thank Zolt\'an K\'ad\'ar and Paul Martin who researched this topic with me.


\newpage
\appendix


\section{Proofs and calculations}

\subsection{Preliminary calculations}
Here we compute a number of relations that are used repeatedly in various proofs.
\begin{calc} \label{calcpppV1}
We consider operators acting on the space $\mathcal{H}_{\tilde{\lambda}}$ where $\tilde{\lambda}=(\lambda_{1},...,\lambda_{L})$. We find that for $\mu,\mu'''\in \mathcal{B}_{\tilde{\lambda}}$ and $1\leq i \leq L-2$ we have:
\begin{align*}
	& \bra{\mu'''} p_{i}^{(\nu_{3})}p_{i+1}^{(\nu_{2})}p_{i}^{(\nu_{1})} \ket{\mu} \\
	& =  \sum_{\mu',\mu''\in\mathcal{B}_{\tilde{\lambda}}} \bra{\mu'''} p_{i}^{(\nu_{3})}\ket{\mu''}\bra{\mu''}p_{i+1}^{(\nu_{2})}\ket{\mu'}\bra{\mu'}p_{i}^{(\nu_{1})} \ket{\mu}\\
	& = \sum_{\mu',\mu''\in\mathcal{B}_{\tilde{\lambda}}} \delta_{\mu_{i+1}}^{\mu_{i+1}'}\delta_{\mu_{i}'}^{\mu_{i}''}\delta_{\mu_{i+1}''}^{\mu_{i+1}'''} \, \left[ \prod_{j\neq i,i+1} \delta_{\mu_{j}}^{\mu_{j}'}\delta_{\mu_{j}}^{\mu_{j}''}\delta_{\mu_{j}}^{\mu_{j}'''} \right] \\
	& \quad\quad \times \left(\bar{F}^{\mu_{i-1}\lambda_{i}\lambda_{i+1}}_{\mu_{i+1}'''}\right)_{\mu_{i}'''}^{\nu_{3}} 
	\left(F^{\mu_{i-1}\lambda_{i}\lambda_{i+1}}_{\mu_{i+1}'''}\right)^{\mu_{i}'}_{\nu_{3}} 
	\left(\bar{F}^{\mu_{i}'\lambda_{i+1}\lambda_{i+2}}_{\mu_{i+2}}\right)_{\mu_{i+1}'''}^{\nu_{2}} 
	\left(F^{\mu_{i}'\lambda_{i+1}\lambda_{i+2}}_{\mu_{i+2}}\right)^{\mu_{i+1}}_{\nu_{2}}
	\left(\bar{F}^{\mu_{i-1}\lambda_{i}\lambda_{i+1}}_{\mu_{i+1}}\right)_{\mu_{i}'}^{\nu_{1}} 
	\left(F^{\mu_{i-1}\lambda_{i}\lambda_{i+1}}_{\mu_{i+1}}\right)^{\mu_{i}}_{\nu_{1}}	\\
	& = \sum_{\upsilon_{1}\in\mathcal{I}} \left[ \prod_{j\neq i,i+1} \delta_{\mu_{j}}^{\mu_{j}'''} \right] 
	\left(\bar{F}^{\mu_{i-1}\lambda_{i}\lambda_{i+1}}_{\mu_{i+1}'''}\right)_{\mu_{i}'''}^{\nu_{3}} 
	\left(\bar{F}^{\upsilon_{1}\lambda_{i+1}\lambda_{i+2}}_{\mu_{i+2}}\right)_{\mu_{i+1}'''}^{\nu_{2}} \\
	& \quad \quad \times
	\left(F^{\mu_{i-1}\lambda_{i}\lambda_{i+1}}_{\mu_{i+1}'''}\right)^{\upsilon_{1}}_{\nu_{3}} 
	\left(\bar{F}^{\mu_{i-1}\lambda_{i}\lambda_{i+1}}_{\mu_{i+1}}\right)_{\upsilon_{1}}^{\nu_{1}} 
	\left(F^{\upsilon_{1}\lambda_{i+1}\lambda_{i+2}}_{\mu_{i+2}}\right)^{\mu_{i+1}}_{\nu_{2}}
	\left(F^{\mu_{i-1}\lambda_{i}\lambda_{i+1}}_{\mu_{i+1}}\right)^{\mu_{i}}_{\nu_{1}} \\
	& = \sum_{\upsilon_{3}\in\mathcal{I}} \left[ \prod_{j\neq i,i+1} \delta_{\mu_{j}}^{\mu_{j}'''} \right] 
	\left(F^{\lambda_{i}\lambda_{i+1}\lambda_{i+2}}_{\upsilon_{3}}\right)^{\nu_{1}}_{\nu_{2}} 
	\left(\bar{F}^{\lambda_{i}\lambda_{i+1}\lambda_{i+2}}_{\upsilon_{3}}\right)^{\nu_{2}}_{\nu_{3}}	\\
	& \quad\quad \times \left(\bar{F}^{\mu_{i-1}\lambda_{i}\lambda_{i+1}}_{\mu_{i+1}'''}\right)_{\mu_{i}'''}^{\nu_{3}} 
	\left(F^{\mu_{i-1}\lambda_{i}\lambda_{i+1}}_{\mu_{i+1}}\right)^{\mu_{i}}_{\nu_{1}} 
	\left(F^{\mu_{i-1}\nu_{1}\lambda_{i+2}}_{\mu_{i+2}}\right)^{\mu_{i+1}}_{\upsilon_{3}} 
	\left(\bar{F}^{\mu_{i-1}\nu_{3}\lambda_{i+2}}_{\mu_{i+2}}\right)^{\upsilon_{3}}_{\mu_{i+1}'''},
\end{align*}
where we have used the relations
\begin{align*}
	\sum_{\upsilon_{2}} \left(F^{\mu_{i-1}\nu_{1}\lambda_{i+2}}_{\mu_{i+2}}\right)^{\mu_{i+1}}_{\upsilon_{2}} \left(F^{\lambda_{i}\lambda_{i+1}\lambda_{i+2}}_{\upsilon_{2}}\right)^{\nu_{1}}_{\nu_{2}} \left(\bar{F}^{\mu_{i-1}\lambda_{i}\nu_{2}}_{\mu_{i+2}}\right)^{\upsilon_{2}}_{\upsilon_{1}} & = \left(\bar{F}^{\mu_{i-1}\lambda_{i}\lambda_{i+1}}_{\mu_{i+1}}\right)^{\nu_{1}}_{\upsilon_{1}} \left(F^{\upsilon_{1}\lambda_{i+1}\lambda_{i+2}}_{\mu_{i+2}}\right)^{\mu_{i+1}}_{\nu_{2}} \\
	\sum_{\upsilon_{3}} \left(F^{\mu_{i-1}\lambda_{i}\nu_{2}}_{\mu_{i+2}}\right)^{\upsilon_{1}}_{\upsilon_{3}} \left(\bar{F}^{\lambda_{i}\lambda_{i+1}\lambda_{i+2}}_{\upsilon_{3}}\right)^{\nu_{2}}_{\nu_{3}} \left(\bar{F}^{\mu_{i-1}\nu_{3}\lambda_{i+2}}_{\mu_{i+2}}\right)^{\upsilon_{3}}_{\mu_{i+1}'''} & = \left(\bar{F}^{\upsilon_{1}\lambda_{i+1}\lambda_{i+2}}_{\mu_{i+2}}\right)^{\nu_{2}}_{\mu_{i+1}'''} \left(F^{\mu_{i-1}\lambda_{i}\lambda_{i+1}}_{\mu_{i+1}'''}\right)^{\upsilon_{1}}_{\nu_{3}} \\
	\sum_{\upsilon_{1}}\left(\bar{F}^{\mu_{i-1}\lambda_{i}\nu_{2}}_{\mu_{i+2}}\right)^{\upsilon_{2}}_{\upsilon_{1}}
	\left(F^{\mu_{i-1}\lambda_{i}\nu_{2}}_{\mu_{i+2}}\right)^{\upsilon_{1}}_{\upsilon_{3}} & = \delta_{\upsilon_{2}}^{\upsilon_{3}} (1-\delta_{N_{\lambda_{i}\nu_{2}}^{\upsilon_{3}}}^{0}) (1-\delta_{N_{\mu_{i-1}\upsilon_{3}}^{\mu_{i+2}}}^{0}).
\end{align*}
\end{calc}

\begin{calc} \label{calcpppV2}
We consider operators acting on the space $\mathcal{H}_{\tilde{\lambda}}$ where $\tilde{\lambda}=(\lambda_{1},...,\lambda_{L})$. We find that for $\mu,\mu'''\in \mathcal{B}_{\tilde{\lambda}}$ and $2\leq i \leq L-1$ we have:
\begin{align*}
	& \bra{\mu'''} p_{i}^{(\nu_{3})}p_{i-1}^{(\nu_{2})}p_{i}^{(\nu_{1})} \ket{\mu} \\
	& =  \sum_{\mu',\mu''\in\mathcal{B}_{\tilde{\lambda}}} \bra{\mu'''} p_{i}^{(\nu_{3})}\ket{\mu''}\bra{\mu''}p_{i-1}^{(\nu_{2})}\ket{\mu'}\bra{\mu'}p_{i}^{(\nu_{1})} \ket{\mu}\\
	& = \sum_{\mu',\mu''\in\mathcal{B}_{\tilde{\lambda}}} \delta_{\mu_{i-1}}^{\mu_{i-1}'}\delta_{\mu_{i}'}^{\mu_{i}''}\delta_{\mu_{i-1}''}^{\mu_{i-1}'''} \, \left[ \prod_{j\neq i-1,i} \delta_{\mu_{j}}^{\mu_{j}'}\delta_{\mu_{j}}^{\mu_{j}''}\delta_{\mu_{j}}^{\mu_{j}'''} \right] \\
	& \quad\quad \times \left(\bar{F}^{\mu_{i-1}''\lambda_{i}\lambda_{i+1}}_{\mu_{i+1}''}\right)_{\mu_{i}'''}^{\nu_{3}} 
	\left(F^{\mu_{i-1}''\lambda_{i}\lambda_{i+1}}_{\mu_{i+1}''}\right)^{\mu_{i}''}_{\nu_{3}} 
	\left(\bar{F}^{\mu_{i-2}'\lambda_{i-1}\lambda_{i}}_{\mu_{i}'}\right)_{\mu_{i-1}''}^{\nu_{2}} 
	\left(F^{\mu_{i-2}'\lambda_{i-1}\lambda_{i}}_{\mu_{i}'}\right)^{\mu_{i-1}'}_{\nu_{2}}
	\left(\bar{F}^{\mu_{i-1}\lambda_{i}\lambda_{i+1}}_{\mu_{i+1}}\right)_{\mu_{i}'}^{\nu_{1}} 
	\left(F^{\mu_{i-1}\lambda_{i}\lambda_{i+1}}_{\mu_{i+1}}\right)^{\mu_{i}}_{\nu_{1}}	\\
	& = \sum_{\upsilon_{1}\in\mathcal{I}} \left[ \prod_{j\neq i-1,i} \delta_{\mu_{j}}^{\mu_{j}'''} \right] \left(\bar{F}^{\mu_{i-1}'''\lambda_{i}\lambda_{i+1}}_{\mu_{i+1}}\right)_{\mu_{i}'''}^{\nu_{3}} \left(F^{\mu_{i-1}\lambda_{i}\lambda_{i+1}}_{\mu_{i+1}}\right)^{\mu_{i}}_{\nu_{1}} \\
	& \quad\quad \times 
	\left(\bar{F}^{\mu_{i-2}\lambda_{i-1}\lambda_{i}}_{\upsilon_{1}}\right)_{\mu_{i-1}'''}^{\nu_{2}} 
	\left(F^{\mu_{i-1}'''\lambda_{i}\lambda_{i+1}}_{\mu_{i+1}}\right)^{\upsilon_{1}}_{\nu_{3}} 
	\left(\bar{F}^{\mu_{i-1}\lambda_{i}\lambda_{i+1}}_{\mu_{i+1}}\right)_{\upsilon_{1}}^{\nu_{1}}
	\left(F^{\mu_{i-2}\lambda_{i-1}\lambda_{i}}_{\upsilon_{1}}\right)^{\mu_{i-1}}_{\nu_{2}} \\
	& = \sum_{\upsilon_{3}\in\mathcal{I}} \left[ \prod_{j\neq i-1,i} \delta_{\mu_{j}}^{\mu_{j}'''} \right] 
	\left(\bar{F}^{\lambda_{i-1}\lambda_{i}\lambda_{i+1}}_{\upsilon_{3}}\right)^{\nu_{1}}_{\nu_{2}} 
	\left(F^{\lambda_{i-1}\lambda_{i}\lambda_{i+1}}_{\upsilon_{3}}\right)^{\nu_{2}}_{\nu_{3}}  \\
	& \quad\quad \times 
	\left(\bar{F}^{\mu_{i-1}'''\lambda_{i}\lambda_{i+1}}_{\mu_{i+1}}\right)_{\mu_{i}'''}^{\nu_{3}} 
	\left(F^{\mu_{i-1}\lambda_{i}\lambda_{i+1}}_{\mu_{i+1}}\right)^{\mu_{i}}_{\nu_{1}} 
	\left(F^{\mu_{i-2}\lambda_{i-1}\nu_{1}}_{\mu_{i+1}}\right)^{\mu_{i-1}}_{\upsilon_{3}} 
	\left(\bar{F}^{\mu_{i-2}\lambda_{i-1}\nu_{3}}_{\mu_{i+1}}\right)^{\upsilon_{3}}_{\mu_{i-1}'''}  
\end{align*}
where we have used the relations
\begin{align*}
	\sum_{\upsilon_{2}} \left(F^{\mu_{i-2}\nu_{2}\lambda_{i+1}}_{\mu_{i+1}}\right)^{\upsilon_{1}}_{\upsilon_{2}} \left(F^{\lambda_{i-1}\lambda_{i}\lambda_{i+1}}_{\upsilon_{2}}\right)^{\nu_{2}}_{\nu_{3}} \left(\bar{F}^{\mu_{i-2}\lambda_{i-1}\nu_{3}}_{\mu_{i+1}}\right)^{\upsilon_{2}}_{\mu_{i-1}'''} & = \left(\bar{F}^{\mu_{i-2}\lambda_{i-1}\lambda_{i}}_{\upsilon_{1}}\right)^{\nu_{2}}_{\mu_{i-1}'''} \left(F^{\mu_{i-1}'''\lambda_{i}\lambda_{i+1}}_{\mu_{i+1}}\right)^{\upsilon_{1}}_{\nu_{3}} \\
	\sum_{\upsilon_{3}} \left(F^{\mu_{i-2}\lambda_{i-1}\nu_{1}}_{\mu_{i+1}}\right)^{\mu_{i-1}}_{\upsilon_{3}} \left(\bar{F}^{\lambda_{i-1}\lambda_{i}\lambda_{i+1}}_{\upsilon_{3}}\right)^{\nu_{1}}_{\nu_{2}} \left(\bar{F}^{\mu_{i-2}\nu_{2}\lambda_{i+1}}_{\mu_{i+1}}\right)^{\upsilon_{3}}_{\upsilon_{1}} & = \left(\bar{F}^{\mu_{i-1}\lambda_{i}\lambda_{i+1}}_{\mu_{i+1}}\right)^{\nu_{1}}_{\upsilon_{1}} \left(F^{\mu_{i-2}\lambda_{i-1}\lambda_{i}}_{\upsilon_{1}}\right)^{\mu_{i-1}}_{\nu_{2}} \\
	\sum_{\upsilon_{1}} \left(\bar{F}^{\mu_{i-2}\nu_{2}\lambda_{i+1}}_{\mu_{i+1}}\right)^{\upsilon_{3}}_{\upsilon_{1}}
	\left(F^{\mu_{i-2}\nu_{2}\lambda_{i+1}}_{\mu_{i+1}}\right)^{\upsilon_{1}}_{\upsilon_{2}}  & = \delta_{\upsilon_{2}}^{\upsilon_{3}} (1-\delta_{N_{\nu_{2}\lambda_{i+1}}^{\upsilon_{3}}}^{0}) (1-\delta_{N_{\mu_{i-2}\upsilon_{3}}^{\mu_{i+1}}}^{0}).
\end{align*}	
\end{calc}

\subsection{Proof of Theorem \ref{ThmGenTLlikeRel}} \label{SecProofGenTLlikeRel}
\begin{calc} \label{calcTLlikeRel1}
We consider operators acting on the space $\mathcal{H}_{\tilde{\lambda}}$ where $\tilde{\lambda}=(\lambda_{1},...,\lambda_{L})$ and $|\mathcal{B}_{\tilde{\lambda}}|\neq 0$, along with some fixed $1\leq i \leq L-2$ and elements $\nu,\nu'\in\mathcal{I}$. Consider the scenario that there exists $c\in\textbf{R}$ such 
\begin{align*}
	\bra{\mu'} p_{i}^{(\nu)}p_{i+1}^{(\nu')}p_{i}^{(\nu)} \ket{\mu} = c \bra{\mu'} p_{i}^{(\nu)} \ket{\mu}
\end{align*}
for all $\mu,\mu'\in \mathcal{B}_{\tilde{\lambda}}$. Using Calculation \ref{calcpppV1} and inversion relations we are able to see that the we must have
\begin{align*}
	c \delta_{\mu_{i+1}}^{\mu_{i+1}'} N_{\mu_{i+1}\lambda_{i+2}}^{\mu_{i+2}} N_{\lambda_{i}\lambda_{i+1}}^{\nu} N_{\mu_{i-1}\nu}^{\mu_{i+1}}
	= 
	\sum_{\upsilon\in\mathcal{I}} \left(F^{\lambda_{i}\lambda_{i+1}\lambda_{i+2}}_{\upsilon}\right)^{\nu}_{\nu'} 
	\left(\bar{F}^{\lambda_{i}\lambda_{i+1}\lambda_{i+2}}_{\upsilon}\right)^{\nu'}_{\nu}
	\left(F^{\mu_{i-1}\nu\lambda_{i+2}}_{\mu_{i+2}}\right)^{\mu_{i+1}}_{\upsilon}  
	\left(\bar{F}^{\mu_{i-1}\nu\lambda_{i+2}}_{\mu_{i+2}}\right)^{\upsilon}_{\mu_{i+1}'} 
\end{align*}
for all $\mu_{i-1},\mu_{i+1},\mu_{i+1}',\mu_{i+2} \in \mathcal{I}$. This can be further simplified to
\begin{align*}
	c N_{\lambda_{i}\lambda_{i+1}}^{\nu} N_{\nu\lambda_{i+2}}^{\upsilon_{1}} N_{\upsilon_{2}\upsilon_{1}}^{\upsilon_{3}}
	= N_{\upsilon_{2}\upsilon_{1}}^{\upsilon_{3}} \left(F^{\lambda_{i}\lambda_{i+1}\lambda_{i+2}}_{\upsilon_{1}}\right)^{\nu}_{\nu'} 	\left(\bar{F}^{\lambda_{i}\lambda_{i+1}\lambda_{i+2}}_{\upsilon_{1}}\right)^{\nu'}_{\nu},
\end{align*}
for all $\upsilon_{1},\upsilon_{2},\upsilon_{3}\in\mathcal{I}$ where $\sum_{\mu_{0}\in\mathcal{I}_{0}} N_{u_{0}\lambda_{1}\cdots\lambda_{i-1}}^{\upsilon_{2}} \geq 1$. \\

\noindent
On the other hand suppose there exists $c\in\textbf{R}$ such that
\begin{align*}
	c N_{\lambda_{i}\lambda_{i+1}}^{\nu} N_{\nu\lambda_{i+2}}^{\upsilon_{1}} N_{\upsilon_{2}\upsilon_{1}}^{\upsilon_{3}}
	= N_{\upsilon_{2}\upsilon_{1}}^{\upsilon_{3}} \left(F^{\lambda_{i}\lambda_{i+1}\lambda_{i+2}}_{\upsilon_{1}}\right)^{\nu}_{\nu'} 	\left(\bar{F}^{\lambda_{i}\lambda_{i+1}\lambda_{i+2}}_{\upsilon_{1}}\right)^{\nu'}_{\nu},
\end{align*}
for all $\upsilon_{1},\upsilon_{2},\upsilon_{3}\in\mathcal{I}$. It follows that for any $\mu,\mu'\in\mathcal{B}_{\tilde{\lambda}}$,
\begin{align*}
	& \bra{\mu'} p_{i}^{(\nu)}p_{i+1}^{(\nu')}p_{i}^{(\nu)} \ket{\mu} \\
	& = \sum_{\upsilon\in\mathcal{I}} \left[ \prod_{j\neq i,i+1} \delta_{\mu_{j}}^{\mu_{j}'} \right] 
	\left(F^{\lambda_{i}\lambda_{i+1}\lambda_{i+2}}_{\upsilon}\right)^{\nu}_{\nu'} 
	\left(\bar{F}^{\lambda_{i}\lambda_{i+1}\lambda_{i+2}}_{\upsilon}\right)^{\nu'}_{\nu}	\\
	& \quad\quad \times \left(\bar{F}^{\mu_{i-1}\lambda_{i}\lambda_{i+1}}_{\mu_{i+1}'}\right)_{\mu_{i}'}^{\nu} 
	\left(F^{\mu_{i-1}\lambda_{i}\lambda_{i+1}}_{\mu_{i+1}}\right)^{\mu_{i}}_{\nu} 
	\left(F^{\mu_{i-1}\nu\lambda_{i+2}}_{\mu_{i+2}}\right)^{\mu_{i+1}}_{\upsilon} 
	\left(\bar{F}^{\mu_{i-1}\nu\lambda_{i+2}}_{\mu_{i+2}}\right)^{\upsilon}_{\mu_{i+1}'} \\
	& = c \sum_{\upsilon\in\mathcal{I}} \left[ \prod_{j\neq i,i+1} \delta_{\mu_{j}}^{\mu_{j}'} \right] 
	\left(\bar{F}^{\mu_{i-1}\lambda_{i}\lambda_{i+1}}_{\mu_{i+1}'}\right)_{\mu_{i}'}^{\nu} 
	\left(F^{\mu_{i-1}\lambda_{i}\lambda_{i+1}}_{\mu_{i+1}}\right)^{\mu_{i}}_{\nu} 
	\left(F^{\mu_{i-1}\nu\lambda_{i+2}}_{\mu_{i+2}}\right)^{\mu_{i+1}}_{\upsilon} 
	\left(\bar{F}^{\mu_{i-1}\nu\lambda_{i+2}}_{\mu_{i+2}}\right)^{\upsilon}_{\mu_{i+1}'} \\
	& = c \left[ \prod_{j\neq i,i+1} \delta_{\mu_{j}}^{\mu_{j}'} \right] \delta_{\mu_{i+1}}^{\mu_{i+1}'} N_{\mu_{i-1}\nu}^{\mu_{i+1}} N_{\mu_{i+1}\lambda_{i+2}}^{\mu_{i+2}}
	\left(\bar{F}^{\mu_{i-1}\lambda_{i}\lambda_{i+1}}_{\mu_{i+1}'}\right)_{\mu_{i}'}^{\nu} 
	\left(F^{\mu_{i-1}\lambda_{i}\lambda_{i+1}}_{\mu_{i+1}}\right)^{\mu_{i}}_{\nu} \\
	& = c \bra{\mu'} p_{i}^{(\nu)} \ket{\mu}.
\end{align*}
Thus we have proven and if and only if concerning solutions of Temperley-Lieb like relations.
\end{calc}

\begin{calc} \label{calcTLlikeRel2}
We consider operators acting on the space $\mathcal{H}_{\tilde{\lambda}}$ where $\tilde{\lambda}=(\lambda_{1},...,\lambda_{L})$ and $|\mathcal{B}_{\tilde{\lambda}}|\neq 0$, along with some fixed $2\leq i \leq L-1$ and elements $\nu,\nu'\in\mathcal{I}$. Consider the scenario that there exists $c\in\textbf{R}$ such 
\begin{align*}
	\bra{\mu'} p_{i}^{(\nu)}p_{i-1}^{(\nu')}p_{i}^{(\nu)} \ket{\mu} = c \bra{\mu'} p_{i}^{(\nu)} \ket{\mu}
\end{align*}
for all $\mu,\mu'\in \mathcal{B}_{\tilde{\lambda}}$. Using Calculation \ref{calcpppV2} and inversion relations we are able to see that the we must have
\begin{align*}
	c \delta_{\mu_{i-1}}^{\mu_{i-1}'} N_{\mu_{i-2}\lambda_{i-1}}^{\mu_{i-1}} N_{\lambda_{i}\lambda_{i+1}}^{\nu} N_{\mu_{i-1}\nu}^{\mu_{i+1}} 
	= \sum_{\upsilon\in\mathcal{I}} \left(\bar{F}^{\lambda_{i-1}\lambda_{i}\lambda_{i+1}}_{\upsilon}\right)^{\nu}_{\nu'} 	\left(F^{\lambda_{i-1}\lambda_{i}\lambda_{i+1}}_{\upsilon}\right)^{\nu'}_{\nu} \left(F^{\mu_{i-2}\lambda_{i-1}\nu}_{\mu_{i+1}}\right)^{\mu_{i-1}}_{\upsilon} 	\left(\bar{F}^{\mu_{i-2}\lambda_{i-1}\nu}_{\mu_{i+1}}\right)^{\upsilon}_{\mu_{i-1}'}  
\end{align*}
for all $\mu_{i-2},\mu_{i-1},\mu_{i-1}',\mu_{i+1}\in\mathcal{I}$. This can be further simplified to
\begin{align*}
	c N_{\lambda_{i}\lambda_{i+1}}^{\nu} N_{\lambda_{i-1}\nu}^{\upsilon_{1}} N_{\upsilon_{2}\upsilon_{1}}^{\upsilon_{3}}  
	= N_{\upsilon_{2}\upsilon_{1}}^{\upsilon_{3}}  \left(\bar{F}^{\lambda_{i-1}\lambda_{i}\lambda_{i+1}}_{\upsilon_{1}}\right)^{\nu}_{\nu'} 		\left(F^{\lambda_{i-1}\lambda_{i}\lambda_{i+1}}_{\upsilon_{1}}\right)^{\nu'}_{\nu} 		
\end{align*}
for all $\upsilon_{1},\upsilon_{2},\upsilon_{3}\in\mathcal{I}$ where $\sum_{\mu_{0}\in\mathcal{I}_{0}} N_{u_{0}\lambda_{1}\cdots\lambda_{i-2}}^{\upsilon_{2}} \geq 1$. \\

\noindent
On the other hand suppose there exists $c\in\textbf{R}$ such that
\begin{align*}
	c N_{\lambda_{i}\lambda_{i+1}}^{\nu} N_{\lambda_{i-1}\nu}^{\upsilon_{1}} N_{\upsilon_{2}\upsilon_{1}}^{\upsilon_{3}}  
	= N_{\upsilon_{2}\upsilon_{1}}^{\upsilon_{3}} \left(\bar{F}^{\lambda_{i-1}\lambda_{i}\lambda_{i+1}}_{\upsilon_{1}}\right)^{\nu}_{\nu'} 		\left(F^{\lambda_{i-1}\lambda_{i}\lambda_{i+1}}_{\upsilon_{1}}\right)^{\nu'}_{\nu} 		
\end{align*}
for all $\upsilon_{1},\upsilon_{2},\upsilon_{3}\in\mathcal{I}$. It follows that for any $\mu,\mu'\in \mathcal{B}_{\tilde{\lambda}}$
\begin{align*}
	& \bra{\mu'} p_{i}^{(\nu)}p_{i-1}^{(\nu')}p_{i}^{(\nu)} \ket{\mu} \\
	& = \sum_{\upsilon\in\mathcal{I}} \left[ \prod_{j\neq i-1,i} \delta_{\mu_{j}}^{\mu_{j}'} \right] 
	\left(\bar{F}^{\lambda_{i-1}\lambda_{i}\lambda_{i+1}}_{\upsilon}\right)^{\nu}_{\nu'} 
	\left(F^{\lambda_{i-1}\lambda_{i}\lambda_{i+1}}_{\upsilon}\right)^{\nu'}_{\nu}  \\
	& \quad\quad \times 
	\left(\bar{F}^{\mu_{i-1}'\lambda_{i}\lambda_{i+1}}_{\mu_{i+1}}\right)_{\mu_{i}'}^{\nu} 
	\left(F^{\mu_{i-1}\lambda_{i}\lambda_{i+1}}_{\mu_{i+1}}\right)^{\mu_{i}}_{\nu} 
	\left(F^{\mu_{i-2}\lambda_{i-1}\nu}_{\mu_{i+1}}\right)^{\mu_{i-1}}_{\upsilon} 
	\left(\bar{F}^{\mu_{i-2}\lambda_{i-1}\nu}_{\mu_{i+1}}\right)^{\upsilon}_{\mu_{i-1}'} \\
	& = c \sum_{\upsilon\in\mathcal{I}} \left[ \prod_{j\neq i-1,i} \delta_{\mu_{j}}^{\mu_{j}'} \right] 
	\left(\bar{F}^{\mu_{i-1}'\lambda_{i}\lambda_{i+1}}_{\mu_{i+1}}\right)_{\mu_{i}'}^{\nu} 
	\left(F^{\mu_{i-1}\lambda_{i}\lambda_{i+1}}_{\mu_{i+1}}\right)^{\mu_{i}}_{\nu} 
	\left(F^{\mu_{i-2}\lambda_{i-1}\nu}_{\mu_{i+1}}\right)^{\mu_{i-1}}_{\upsilon} 
	\left(\bar{F}^{\mu_{i-2}\lambda_{i-1}\nu}_{\mu_{i+1}}\right)^{\upsilon}_{\mu_{i-1}'} \\
	& = c \sum_{\upsilon\in\mathcal{I}} \left[ \prod_{j\neq i-1,i} \delta_{\mu_{j}}^{\mu_{j}'} \right] \delta_{\mu_{i-1}}^{\mu_{i-1}'} N_{\mu_{i-2}\lambda_{i-1}}^{\mu_{i-1}} N_{\mu_{i-1}\nu}^{\mu_{i+1}} \left(\bar{F}^{\mu_{i-1}'\lambda_{i}\lambda_{i+1}}_{\mu_{i+1}}\right)_{\mu_{i}'}^{\nu} 	\left(F^{\mu_{i-1}\lambda_{i}\lambda_{i+1}}_{\mu_{i+1}}\right)^{\mu_{i}}_{\nu} \\
	& = c \bra{\mu'}p_{i}^{(\nu)} \ket{\mu}
\end{align*}
\end{calc}


\subsection{BMW like relations}
\begin{calc} \label{calcRRppp1}
We wish to prove that if $\nu$ acts one-dimensionally on $\lambda$ and there exists $c\in\R^{\times}$ such that
\begin{align*}
  c = \left(F^{\lambda\lambda\lambda}_{\phi_{\nu}(\lambda)}\right)^{\nu}_{\nu} \left(\bar{F}^{\lambda\lambda\lambda}_{\phi_{\nu}(\lambda)}\right)^{\nu}_{\nu},
\end{align*}
then
\begin{align*}
	R_{i}R_{i+1}p_{i}^{(\nu)} = c^{-1} \left[ R^{\lambda\nu}_{\phi_{\nu}(\lambda)} \left(\bar{F}^{\lambda\lambda\lambda}_{\phi_{\nu}(\lambda)}\right)^{\nu}_{\nu}  \right] p_{i+1}^{(\nu)}p_{i}^{(\nu)}, \\
	R_{i}^{-1}R_{i+1}^{-1}p_{i}^{(\nu)} = c^{-1} \left[ \bar{R}^{\lambda\nu}_{\phi_{\nu}(\lambda)} \left(\bar{F}^{\lambda\lambda\lambda}_{\phi_{\nu}(\lambda)}\right)^{\nu}_{\nu} \right] p_{i+1}^{(\nu)}p_{i}^{(\nu)}, \\
	p_{i+1}^{(\nu)}R_{i}R_{i+1} = c^{-1} \left[ R^{\lambda\nu}_{\phi_{\nu}(\lambda)} \left(\bar{F}^{\lambda\lambda\lambda}_{\phi_{\nu}(\lambda)}\right)^{\nu}_{\nu}  \right] p_{i+1}^{(\nu)}p_{i}^{(\nu)}, \\
	p_{i+1}^{(\nu)}R_{i}^{-1}R_{i+1}^{-1} = c^{-1} \left[ \bar{R}^{\lambda\nu}_{\phi_{\nu}(\lambda)} \left(\bar{F}^{\lambda\lambda\lambda}_{\phi_{\nu}(\lambda)}\right)^{\nu}_{\nu} \right] p_{i+1}^{(\nu)}p_{i}^{(\nu)}.
\end{align*}
We first compute the product of projection operators on the right hand side of the above relations,
\begin{align*}
	& \bra{\mu'''} p_{i+1}^{(\nu)}p_{i}^{(\nu)} \ket{\mu} \\
	& = \sum_{\mu',\mu''\in\mathcal{B}_{\tilde{\lambda}}} \bra{\mu'''} p_{i+1}^{(\nu)}\ket{\mu'}\bra{\mu'}p_{i}^{(\nu)} \ket{\mu}\\
	& = \sum_{\mu',\mu''\in\mathcal{B}_{\tilde{\lambda}}} \delta_{\mu_{i+1}}^{\mu_{i+1}'}\delta_{\mu_{i}'}^{\mu_{i}'''} \, \left[ \prod_{j\neq i,i+1} \delta_{\mu_{j}}^{\mu_{j}'}\delta_{\mu_{j}}^{\mu_{j}'''} \right] 
	\times \left(\bar{F}^{\mu_{i-1}\lambda\lambda}_{\mu_{i+1}}\right)_{\mu_{i}'''}^{\nu} \left(F^{\mu_{i-1}\lambda\lambda}_{\mu_{i+1}}\right)^{\mu_{i}}_{\nu}
	\left(\bar{F}^{\mu_{i}'''\lambda\lambda}_{\mu_{i+2}}\right)_{\mu_{i+1}'''}^{\nu} \left(F^{\mu_{i}'''\lambda\lambda}_{\mu_{i+2}}\right)^{\mu_{i+1}}_{\nu} \\
	& = \left[ \prod_{j\neq i,i+1}\delta_{\mu_{j}}^{\mu_{j}'''} \right] 
	\times \left(\bar{F}^{\mu_{i-1}\lambda\lambda}_{\mu_{i+1}}\right)_{\mu_{i}'''}^{\nu} \left(F^{\mu_{i-1}\lambda\lambda}_{\mu_{i+1}}\right)^{\mu_{i}}_{\nu}
	\left(\bar{F}^{\mu_{i}'''\lambda\lambda}_{\mu_{i+2}}\right)_{\mu_{i+1}'''}^{\nu} \left(F^{\mu_{i}'''\lambda\lambda}_{\mu_{i+2}}\right)^{\mu_{i+1}}_{\nu}. \\
\end{align*}
Next, we compute the left hand side of the top relation with $\mu,\mu'''\in\mathcal{B}_{\tilde{\lambda}}$,
\begin{align*}
	& \bra{\mu'''} R_{i}R_{i+1}p_{i}^{(\nu)} \ket{\mu} \\
	& = \sum_{\nu',\nu''\in\mathcal{I}} \sum_{\mu',\mu''\in\mathcal{B}_{\tilde{\lambda}}} R^{\lambda\lambda}_{\nu''}R^{\lambda\lambda}_{\nu'} \bra{\mu'''} p_{i}^{(\nu'')}\ket{\mu''}\bra{\mu''}p_{i+1}^{(\nu')}\ket{\mu'}\bra{\mu'}p_{i}^{(\nu)} \ket{\mu}\\
	& = \sum_{\upsilon_{3},\nu',\nu''\in\mathcal{I}} \left[ \prod_{j\neq i,i+1} \delta_{\mu_{j}}^{\mu_{j}'''} \right] R^{\lambda\lambda}_{\nu''}R^{\lambda\lambda}_{\nu'}
	\left(F^{\lambda\lambda\lambda}_{\upsilon_{3}}\right)^{\nu}_{\nu'} 
	\left(\bar{F}^{\lambda\lambda\lambda}_{\upsilon_{3}}\right)^{\nu'}_{\nu''}	
	\left(\bar{F}^{\mu_{i-1}\lambda\lambda}_{\mu_{i+1}'''}\right)_{\mu_{i}'''}^{\nu''} 
	\left(F^{\mu_{i-1}\lambda\lambda}_{\mu_{i+1}}\right)^{\mu_{i}}_{\nu} 
	\left(F^{\mu_{i-1}\nu\lambda}_{\mu_{i+2}}\right)^{\mu_{i+1}}_{\upsilon_{3}} 
	\left(\bar{F}^{\mu_{i-1}\nu''\lambda}_{\mu_{i+2}}\right)^{\upsilon_{3}}_{\mu_{i+1}'''} \\
	& = R^{\lambda\nu}_{\phi_{\nu}(\lambda)} \sum_{\upsilon_{3},\nu''\in\mathcal{I}} \left[ \prod_{j\neq i,i+1} \delta_{\mu_{j}}^{\mu_{j}'''} \right] \left(\bar{F}^{\lambda\lambda\lambda}_{\upsilon_{3}}\right)^{\nu}_{\nu''} \left(\bar{F}^{\mu_{i-1}\lambda\lambda}_{\mu_{i+1}'''}\right)_{\mu_{i}'''}^{\nu''} 
	\left(F^{\mu_{i-1}\lambda\lambda}_{\mu_{i+1}}\right)^{\mu_{i}}_{\nu} 
	\left(F^{\mu_{i-1}\nu\lambda}_{\mu_{i+2}}\right)^{\mu_{i+1}}_{\upsilon_{3}} 
	\left(\bar{F}^{\mu_{i-1}\nu''\lambda}_{\mu_{i+2}}\right)^{\upsilon_{3}}_{\mu_{i+1}'''} \\
	& = R^{\lambda\nu}_{\phi_{\nu}(\lambda)} \sum_{\upsilon_{3}\in\mathcal{I}} \left[ \prod_{j\neq i,i+1} \delta_{\mu_{j}}^{\mu_{j}'''} \right] 
	\left(F^{\mu_{i-1}\lambda\lambda}_{\mu_{i+1}}\right)^{\mu_{i}}_{\nu} \left(\bar{F}^{\mu_{i}'''\lambda\lambda}_{\mu_{i+2}}\right)^{\nu}_{\mu_{i+1}'''}
	\left(F^{\mu_{i-1}\nu\lambda}_{\mu_{i+2}}\right)^{\mu_{i+1}}_{\upsilon_{3}}
	\left(\bar{F}^{\mu_{i-1}\lambda\nu}_{\mu_{i+2}}\right)^{\upsilon_{3}}_{\mu_{i}'''} \\
	& = c^{-1} R^{\lambda\nu}_{\phi_{\nu}(\lambda)} \sum_{\upsilon_{3},\upsilon_{4}\in\mathcal{I}} \left[ \prod_{j\neq i,i+1} \delta_{\mu_{j}}^{\mu_{j}'''} \right] \delta_{\upsilon_{4}}^{\mu_{i}'''} 
	\left(F^{\mu_{i-1}\lambda\lambda}_{\mu_{i+1}}\right)^{\mu_{i}}_{\nu} \left(\bar{F}^{\mu_{i}'''\lambda\lambda}_{\mu_{i+2}}\right)^{\nu}_{\mu_{i+1}'''}
	\left(\bar{F}^{\mu_{i-1}\lambda\lambda}_{\mu_{i+1}}\right)^{\nu}_{\upsilon_{4}} 
	\left(F^{\upsilon_{4}\lambda\lambda}_{\mu_{i+2}}\right)^{\mu_{i+1}}_{\nu}
	\left(\bar{F}^{\lambda\lambda\lambda}_{\upsilon_{3}}\right)^{\nu}_{\nu} \\
	& = c^{-1} \left[R^{\lambda\nu}_{\phi_{\nu}(\lambda)} \left(\bar{F}^{\lambda\lambda\lambda}_{\phi_{\nu}(\lambda)}\right)^{\nu}_{\nu}  \right] \times \bra{\mu'''} p_{i+1}^{(\nu)}p_{i}^{(\nu)} \ket{\mu}.
\end{align*}
In the above computation we have used Calculation \ref{calcpppV1} and the following special cases of the axioms of braided tensor systems,
\begin{align*}
	\sum_{\nu'\in\mathcal{I}} \left(F^{\lambda\lambda\lambda}_{\upsilon_{3}}\right)^{\nu}_{\nu'} R^{\lambda\lambda}_{\nu'} \left(\bar{F}^{\lambda\lambda\lambda}_{\upsilon_{3}}\right)^{\nu'}_{\nu''} R^{\lambda\lambda}_{\nu''} & = R^{\lambda\nu}_{\upsilon_{3}} \left(\bar{F}^{\lambda\lambda\lambda}_{\upsilon_{3}}\right)^{\nu}_{\nu''}, \\
	\sum_{\nu''} \left(\bar{F}^{\lambda\lambda\lambda}_{\upsilon_{3}}\right)^{\nu}_{\nu''}
	\left(\bar{F}^{\mu_{i-1}\nu''\lambda}_{\mu_{i+2}}\right)^{\upsilon_{3}}_{\mu_{i+1}'''} 
	\left(\bar{F}^{\mu_{i-1}\lambda\lambda}_{\mu_{i+1}'''}\right)^{\nu''}_{\mu_{i}'''} 
	& =  \left(\bar{F}^{\mu_{i-1}\lambda\nu}_{\mu_{i+2}}\right)^{\upsilon_{3}}_{\mu_{i}'''} \left(\bar{F}^{\mu_{i}'''\lambda\lambda}_{\mu_{i+2}}\right)^{\nu}_{\mu_{i+1}'''}, \\
	\sum_{\upsilon_{4}} \left(\bar{F}^{\mu_{i-1}\lambda\lambda}_{\mu_{i+1}}\right)^{\nu}_{\upsilon_{4}} \left(F^{\upsilon_{4}\lambda\lambda}_{\mu_{i+2}}\right)^{\mu_{i+1}}_{\nu} \left(F^{\mu_{i-1}\lambda\nu}_{\mu_{i+2}}\right)^{\upsilon_{4}}_{\upsilon_{3}} & = \left(F^{\mu_{i-1}\nu\lambda}_{\mu_{i+2}}\right)^{\mu_{i+1}}_{\upsilon_{3}} \left(F^{\lambda\lambda\lambda}_{\upsilon_{3}}\right)^{\nu}_{\nu}, \\
	\left(F^{\mu_{i-1}\lambda\nu}_{\mu_{i+2}}\right)^{\upsilon_{4}}_{\phi_{\nu}(\lambda)}
	\left(\bar{F}^{\mu_{i-1}\lambda\nu}_{\mu_{i+2}}\right)^{\phi_{\nu}(\lambda)}_{\mu_{i}'''} & = \delta_{\upsilon_{4}}^{\mu_{i}'''}  (1-\delta^{0}_{N_{\mu_{i}'''\nu}^{\mu_{i+2}}}) (1-\delta^{0}_{N_{\mu_{i-1}\lambda}^{\mu_{i}'''}}).
\end{align*}
We now investigate the second relation and calculate that for $\mu,\mu'''\in\mathcal{B}_{\tilde{\lambda}}$ we have,
\begin{align*}
	& \bra{\mu'''} R_{i}^{-1}R_{i+1}^{-1}p_{i}^{(\nu)} \ket{\mu} \\
	& = \sum_{\nu',\nu''\in\mathcal{I}} \sum_{\mu',\mu''\in\mathcal{B}_{\tilde{\lambda}}} \bar{R}^{\lambda\lambda}_{\nu''} \bar{R}^{\lambda\lambda}_{\nu'} \bra{\mu'''} p_{i}^{(\nu'')}\ket{\mu''}\bra{\mu''}p_{i+1}^{(\nu')}\ket{\mu'}\bra{\mu'}p_{i}^{(\nu)} \ket{\mu}\\
	& = \sum_{\nu',\nu'',\upsilon_{3}\in\mathcal{I}} \left[ \prod_{j\neq i,i+1} \delta_{\mu_{j}}^{\mu_{j}'''} \right] \bar{R}^{\lambda\lambda}_{\nu'} \bar{R}^{\lambda\lambda}_{\nu''} \left(F^{\lambda\lambda\lambda}_{\upsilon_{3}}\right)^{\nu}_{\nu'}  \left(\bar{F}^{\lambda\lambda\lambda}_{\upsilon_{3}}\right)^{\nu'}_{\nu''} 
	\left(F^{\mu_{i-1}\nu\lambda}_{\mu_{i+2}}\right)^{\mu_{i+1}}_{\upsilon_{3}}
	\left(F^{\mu_{i-1}\lambda\lambda}_{\mu_{i+1}}\right)^{\mu_{i}}_{\nu}
	\left(\bar{F}^{\mu_{i-1}\lambda\lambda}_{\mu_{i+1}'''}\right)_{\mu_{i}'''}^{\nu''}
	\left(\bar{F}^{\mu_{i-1}\nu''\lambda}_{\mu_{i+2}}\right)^{\upsilon_{3}}_{\mu_{i+1}'''} \\
	& = \bar{R}^{\lambda\nu}_{\phi_{\nu}(\lambda)}  \sum_{\nu'',\upsilon_{3}\in\mathcal{I}} \left[ \prod_{j\neq i,i+1} \delta_{\mu_{j}}^{\mu_{j}'''} \right] 
	\left(F^{\mu_{i-1}\nu\lambda}_{\mu_{i+2}}\right)^{\mu_{i+1}}_{\upsilon_{3}}
	\left(F^{\mu_{i-1}\lambda\lambda}_{\mu_{i+1}}\right)^{\mu_{i}}_{\nu}
	\left(\bar{F}^{\lambda\lambda\lambda}_{\upsilon_{3}}\right)^{\nu}_{\nu''}
	\left(\bar{F}^{\mu_{i-1}\nu''\lambda}_{\mu_{i+2}}\right)^{\upsilon_{3}}_{\mu_{i+1}'''}
	\left(\bar{F}^{\mu_{i-1}\lambda\lambda}_{\mu_{i+1}'''}\right)_{\mu_{i}'''}^{\nu''} \\
	& = c^{-1} \left[ \bar{R}^{\lambda\nu}_{\phi_{\nu}(\lambda)} \left(\bar{F}^{\lambda\lambda\lambda}_{\phi_{\nu}(\lambda)}\right)^{\nu}_{\nu} \right] \times \bra{\mu'''} p_{i+1}^{(\nu)}p_{i}^{(\nu)} \ket{\mu}.
\end{align*}
In this computation we have used Calculation \ref{calcpppV1}, relations found in the previous computation of this calculation as well as the following special case of the hexagon relation,
\begin{align*}
	\sum_{\nu'\in\mathcal{I}} \left(F^{\lambda\lambda\lambda}_{\upsilon_{3}}\right)^{\nu}_{\nu'} \bar{R}^{\lambda\lambda}_{\nu'} \left(\bar{F}^{\lambda\lambda\lambda}_{\upsilon_{3}}\right)^{\nu'}_{\nu''} \bar{R}^{\lambda\lambda}_{\nu''} & = \bar{R}^{\lambda\nu}_{\upsilon_{3}} \left(\bar{F}^{\lambda\lambda\lambda}_{\upsilon_{3}}\right)^{\nu}_{\nu''}.
\end{align*}
We now consider the left hand side of the third relation and calculate that for $\mu,\mu'''\in\mathcal{B}_{\tilde{\lambda}}$ we have,
\begin{align*}
	& \bra{\mu'''} p_{i+1}^{(\nu)}R_{i}R_{i+1} \ket{\mu} \\
	& = \sum_{\nu',\nu''\in\mathcal{I}} \sum_{\mu',\mu''\in\mathcal{B}_{\tilde{\lambda}}} R^{\lambda\lambda}_{\nu''}R^{\lambda\lambda}_{\nu'} \bra{\mu'''} p_{i+1}^{(\nu)}\ket{\mu''}\bra{\mu''}p_{i}^{(\nu')}\ket{\mu'}\bra{\mu'}p_{i+1}^{(\nu'')} \ket{\mu}\\
	& = \sum_{\upsilon_{3},\nu',\nu''\in\mathcal{I}} \left[ \prod_{j\neq i,i+1} \delta_{\mu_{j}}^{\mu_{j}'''} \right] 
	R^{\lambda\lambda}_{\nu''}R^{\lambda\lambda}_{\nu'} 
	\left(\bar{F}^{\lambda\lambda\lambda}_{\upsilon_{3}}\right)^{\nu''}_{\nu'} 
	\left(F^{\lambda\lambda\lambda}_{\upsilon_{3}}\right)^{\nu'}_{\nu}
	\left(\bar{F}^{\mu_{i}'''\lambda\lambda}_{\mu_{i+2}}\right)_{\mu_{i+1}'''}^{\nu} 
	\left(F^{\mu_{i}\lambda\lambda}_{\mu_{i+2}}\right)^{\mu_{i+1}}_{\nu''} 
	\left(F^{\mu_{i-1}\lambda\nu''}_{\mu_{i+2}}\right)^{\mu_{i}}_{\upsilon_{3}} 
	\left(\bar{F}^{\mu_{i-1}\lambda\nu}_{\mu_{i+2}}\right)^{\upsilon_{3}}_{\mu_{i}'''} \\
	& = R^{\lambda\nu}_{\phi_{\nu}(\lambda)} \sum_{\upsilon_{3},\nu''\in\mathcal{I}} \left[\prod_{j\neq i,i+1}\delta_{\mu_{j}}^{\mu_{j}'''} \right] 
	\left(\bar{F}^{\lambda\lambda\lambda}_{\upsilon_{3}}\right)^{\nu''}_{\nu} 
	\left(\bar{F}^{\mu_{i}'''\lambda\lambda}_{\mu_{i+2}}\right)_{\mu_{i+1}'''}^{\nu} 
	\left(F^{\mu_{i}\lambda\lambda}_{\mu_{i+2}}\right)^{\mu_{i+1}}_{\nu''} 
	\left(F^{\mu_{i-1}\lambda\nu''}_{\mu_{i+2}}\right)^{\mu_{i}}_{\upsilon_{3}} 
	\left(\bar{F}^{\mu_{i-1}\lambda\nu}_{\mu_{i+2}}\right)^{\upsilon_{3}}_{\mu_{i}'''} \\
	& = R^{\lambda\nu}_{\phi_{\nu}(\lambda)} \sum_{\upsilon_{3}\in\mathcal{I}} \left[ \prod_{j\neq i,i+1} \delta_{\mu_{j}}^{\mu_{j}'''} \right] 
	\left(\bar{F}^{\mu_{i}'''\lambda\lambda}_{\mu_{i+2}}\right)_{\mu_{i+1}'''}^{\nu} 
	\left(\bar{F}^{\mu_{i-1}\lambda\nu}_{\mu_{i+2}}\right)^{\upsilon_{3}}_{\mu_{i}'''} 
	\left(F^{\mu_{i-1}\lambda\lambda}_{\mu_{i+1}}\right)^{\mu_{i}}_{\nu} 
	\left(F^{\mu_{i-1}\nu\lambda}_{\mu_{i+2}}\right)^{\mu_{i+1}}_{\upsilon_{3}} \\
	& = c^{-1} R^{\lambda\nu}_{\phi_{\nu}(\lambda)}  \sum_{\upsilon_{3},\upsilon_{4}\in\mathcal{I}} \left[ \prod_{j\neq i,i+1} \delta_{\mu_{j}}^{\mu_{j}'''} \right] \delta_{\mu_{i}'''}^{\upsilon_{4}} 
	\left(\bar{F}^{\mu_{i}'''\lambda\lambda}_{\mu_{i+2}}\right)_{\mu_{i+1}'''}^{\nu} 
	\left(F^{\mu_{i-1}\lambda\lambda}_{\mu_{i+1}}\right)^{\mu_{i}}_{\nu} 
	\left(\bar{F}^{\mu_{i-1}\lambda\lambda}_{\mu_{i+1}}\right)^{\nu}_{\upsilon_{4}} 
	\left(F^{\upsilon_{4}\lambda\lambda}_{\mu_{i+2}}\right)^{\mu_{i+1}}_{\nu} 
	\left(\bar{F}^{\lambda\lambda\lambda}_{\upsilon_{3}}\right)^{\nu}_{\nu} \\
	& = c^{-1} \left[ R^{\lambda\nu}_{\phi_{\nu}(\lambda)} \left(\bar{F}^{\lambda\lambda\lambda}_{\phi_{\nu}(\lambda)}\right)^{\nu}_{\nu} \right] \times \bra{\mu'''} p_{i+1}^{(\nu)}p_{i}^{(\nu)} \ket{\mu}.
\end{align*}
In this computation we have used Calculation \ref{calcpppV2}, relations found in the previous computation of this calculation as well as the following special cases of the axioms of the braided tensor system,
\begin{align*}
	\sum_{\nu'\in\mathcal{I}} R^{\lambda\lambda}_{\nu''} \left(\bar{F}^{\lambda\lambda\lambda}_{\upsilon_{3}}\right)^{\nu''}_{\nu'} R^{\lambda\lambda}_{\nu'} \left(F^{\lambda\lambda\lambda}_{\upsilon_{3}}\right)^{\nu'}_{\nu} & = \left(\bar{F}^{\lambda\lambda\lambda}_{\upsilon_{3}}\right)^{\nu''}_{\nu} R^{\lambda\nu}_{\upsilon_{3}}, \\
	\sum_{\nu''\in\mathcal{I}}  \left(F^{\mu_{i}\lambda\lambda}_{\mu_{i+2}}\right)^{\mu_{i+1}}_{\nu''} \left(F^{\mu_{i-1}\lambda\nu''}_{\mu_{i+2}}\right)^{\mu_{i}}_{\upsilon_{3}}\left(\bar{F}^{\lambda\lambda\lambda}_{\upsilon_{3}}\right)^{\nu''}_{\nu} & = \left(F^{\mu_{i-1}\lambda\lambda}_{\mu_{i+1}}\right)^{\mu_{i}}_{\nu} \left(F^{\mu_{i-1}\nu\lambda}_{\mu_{i+2}}\right)^{\mu_{i+1}}_{\upsilon_{3}}, \\
	\sum_{\upsilon_{4}\mathcal{I}} \left(\bar{F}^{\mu_{i-1}\lambda\lambda}_{\mu_{i+1}}\right)^{\nu}_{\upsilon_{4}} \left(F^{\upsilon_{4}\lambda\lambda}_{\mu_{i+2}}\right)^{\mu_{i+1}}_{\nu} \left(F^{\mu_{i-1}\lambda\nu}_{\mu_{i+2}}\right)^{\upsilon_{4}}_{\upsilon_{3}} & = \left(F^{\mu_{i-1}\nu\lambda}_{\mu_{i+2}}\right)^{\mu_{i+1}}_{\upsilon_{3}} \left(F^{\lambda\lambda\lambda}_{\upsilon_{3}}\right)^{\nu}_{\nu}, \\
	\left(\bar{F}^{\mu_{i-1}\lambda\nu}_{\mu_{i+2}}\right)^{\phi_{\nu}(\lambda)}_{\mu_{i}'''} \left(F^{\mu_{i-1}\lambda\nu}_{\mu_{i+2}}\right)^{\upsilon_{4}}_{\phi_{\nu}(\lambda)} & = \delta_{\mu_{i}'''}^{\upsilon_{4}} (1-\delta_{N_{\mu_{i-1}\lambda}^{\upsilon_{4}}}^{0}) (1-\delta_{N_{\upsilon_{4}\nu}^{\mu_{i+2}}}^{0}).
\end{align*}
We now consider the left hand side of the fourth relation and calculate that for $\mu,\mu'''\in\mathcal{B}_{\tilde{\lambda}}$ we have,
\begin{align*}
	& \bra{\mu'''} p_{i+1}^{(\nu)}R_{i}^{-1}R_{i+1}^{-1} \ket{\mu} \\
	& = \sum_{\nu',\nu''\in\mathcal{I}} \sum_{\mu',\mu''\in\mathcal{B}_{\tilde{\lambda}}} \bar{R}^{\lambda\lambda}_{\nu''}\bar{R}^{\lambda\lambda}_{\nu'} \bra{\mu'''} p_{i+1}^{(\nu)}\ket{\mu''}\bra{\mu''}p_{i}^{(\nu')}\ket{\mu'}\bra{\mu'}p_{i+1}^{(\nu'')} \ket{\mu}\\
	& = \sum_{\upsilon_{3},\nu',\nu''\in\mathcal{I}} \left[ \prod_{j\neq i,i+1} \delta_{\mu_{j}}^{\mu_{j}'''} \right] 
	\bar{R}^{\lambda\lambda}_{\nu''}\bar{R}^{\lambda\lambda}_{\nu'} \left(\bar{F}^{\lambda\lambda\lambda}_{\upsilon_{3}}\right)^{\nu''}_{\nu'} 
	\left(F^{\lambda\lambda\lambda}_{\upsilon_{3}}\right)^{\nu'}_{\nu}
	\left(\bar{F}^{\mu_{i}'''\lambda\lambda}_{\mu_{i+2}}\right)_{\mu_{i+1}'''}^{\nu} 
	\left(F^{\mu_{i}\lambda\lambda}_{\mu_{i+2}}\right)^{\mu_{i+1}}_{\nu''} 
	\left(F^{\mu_{i-1}\lambda\nu''}_{\mu_{i+2}}\right)^{\mu_{i}}_{\upsilon_{3}} 
	\left(\bar{F}^{\mu_{i-1}\lambda\nu}_{\mu_{i+2}}\right)^{\upsilon_{3}}_{\mu_{i}'''} \\
	& = \bar{R}^{\lambda\nu}_{\phi_{\nu}(\lambda)}  \sum_{\upsilon_{3},\nu''\in\mathcal{I}} \left[ \prod_{j\neq i,i+1} \delta_{\mu_{j}}^{\mu_{j}'''} \right] 
	\left(\bar{F}^{\lambda\lambda\lambda}_{\upsilon_{3}}\right)^{\nu''}_{\nu} 
	\left(\bar{F}^{\mu_{i}'''\lambda\lambda}_{\mu_{i+2}}\right)_{\mu_{i+1}'''}^{\nu} 
	\left(F^{\mu_{i}\lambda\lambda}_{\mu_{i+2}}\right)^{\mu_{i+1}}_{\nu''} 
	\left(F^{\mu_{i-1}\lambda\nu''}_{\mu_{i+2}}\right)^{\mu_{i}}_{\upsilon_{3}} 
	\left(\bar{F}^{\mu_{i-1}\lambda\nu}_{\mu_{i+2}}\right)^{\upsilon_{3}}_{\mu_{i}'''} \\
	& = c^{-1} \left[ \bar{R}^{\lambda\nu}_{\phi_{\nu}(\lambda)} \left(\bar{F}^{\lambda\lambda\lambda}_{\phi_{\nu}(\lambda)}\right)^{\nu}_{\nu} \right] \times \bra{\mu'''} p_{i+1}^{(\nu)}p_{i}^{(\nu)} \ket{\mu}.
\end{align*}
In this computation we have used Calculation \ref{calcpppV2}, relations found in the previous computation of this calculation as well as the following special case of the hexagon relation,
\begin{align*}
	\sum_{\nu'\in\mathcal{I}} \bar{R}^{\lambda\lambda}_{\nu''} \left(\bar{F}^{\lambda\lambda\lambda}_{\upsilon_{3}}\right)^{\nu''}_{\nu'} \bar{R}^{\lambda\lambda}_{\nu'} \left(F^{\lambda\lambda\lambda}_{\upsilon_{3}}\right)^{\nu'}_{\nu} 
& = \left(\bar{F}^{\lambda\lambda\lambda}_{\upsilon_{3}}\right)^{\nu''}_{\nu} \bar{R}^{\lambda\nu}_{\upsilon_{3}}. 
\end{align*}
\end{calc}

\begin{calc} \label{calcRRppp2}
We wish to prove that if $\nu$ acts one-dimensionally on $\lambda$ and there exists $c\in\R^{\times}$ such that
\begin{align*}
  c = \left(F^{\lambda\lambda\lambda}_{\phi_{\nu}(\lambda)}\right)^{\nu}_{\nu} \left(\bar{F}^{\lambda\lambda\lambda}_{\phi_{\nu}(\lambda)}\right)^{\nu}_{\nu},
\end{align*}
then
\begin{align*}
	R_{i}R_{i-1}p_{i}^{(\nu)} & = c^{-1} \left[	R^{\nu\lambda}_{\phi_{\nu}(\lambda)} \left(F^{\lambda\lambda\lambda}_{\phi_{\nu}(\lambda)}\right)^{\nu}_{\nu} \right] p_{i-1}^{(\nu)}p_{i}^{(\nu)} \\
	R_{i}^{-1}R_{i-1}^{-1}p_{i}^{(\nu)} & = c^{-1} \left[ \bar{R}^{\nu\lambda}_{\phi_{\nu}(\lambda)} \left(F^{\lambda\lambda\lambda}_{\phi_{\nu}(\lambda)}\right)^{\nu}_{\nu} \right] p_{i-1}^{(\nu)}p_{i}^{(\nu)} \\
	p_{i-1}^{(\nu)}R_{i}R_{i-1} & = c^{-1} \left[	R^{\nu\lambda}_{\phi_{\nu}(\lambda)} \left(F^{\lambda\lambda\lambda}_{\phi_{\nu}(\lambda)}\right)^{\nu}_{\nu} \right] p_{i-1}^{(\nu)}p_{i}^{(\nu)} \\
	p_{i-1}^{(\nu)}R_{i}^{-1}R_{i-1}^{-1} & = c^{-1} \left[ \bar{R}^{\nu\lambda}_{\phi_{\nu}(\lambda)} \left(F^{\lambda\lambda\lambda}_{\phi_{\nu}(\lambda)}\right)^{\nu}_{\nu} \right] p_{i-1}^{(\nu)}p_{i}^{(\nu)}
\end{align*}
We first find
\begin{align*}
	& \bra{\mu'''} p_{i-1}^{(\nu)}p_{i}^{(\nu)} \ket{\mu} \\
	& = \sum_{\mu',\mu''\in\mathcal{B}_{\tilde{\lambda}}} \bra{\mu'''} p_{i-1}^{(\nu)}\ket{\mu'}\bra{\mu'}p_{i}^{(\nu)} \ket{\mu}\\
	& = \sum_{\mu',\mu''\in\mathcal{B}_{\tilde{\lambda}}} \delta_{\mu_{i-1}}^{\mu_{i-1}'}\delta_{\mu_{i}'}^{\mu_{i}'''} \, \left[ \prod_{j\neq i-1,i}
	 \delta_{\mu_{j}}^{\mu_{j}'}\delta_{\mu_{j}}^{\mu_{j}'''} \right] 
	\times \left(\bar{F}^{\mu_{i-2}\lambda\lambda}_{\mu_{i}'''}\right)_{\mu_{i-1}'''}^{\nu} \left(F^{\mu_{i-2}\lambda\lambda}_{\mu_{i}'''}\right)^{\mu_{i-1}}_{\nu} 
	\left(\bar{F}^{\mu_{i-1}\lambda\lambda}_{\mu_{i+1}}\right)_{\mu_{i}'''}^{\nu} \left(F^{\mu_{i-1}\lambda\lambda}_{\mu_{i+1}}\right)^{\mu_{i}}_{\nu} \\
	& = \left[ \prod_{j\neq i-1,i}	\delta_{\mu_{j}}^{\mu_{j}'''} \right] 
	\times \left(\bar{F}^{\mu_{i-2}\lambda\lambda}_{\mu_{i}'''}\right)_{\mu_{i-1}'''}^{\nu} \left(F^{\mu_{i-2}\lambda\lambda}_{\mu_{i}'''}\right)^{\mu_{i-1}}_{\nu} 
	\left(\bar{F}^{\mu_{i-1}\lambda\lambda}_{\mu_{i+1}}\right)_{\mu_{i}'''}^{\nu} \left(F^{\mu_{i-1}\lambda\lambda}_{\mu_{i+1}}\right)^{\mu_{i}}_{\nu}. \\
\end{align*}
Next, we compute the LHS of the top relation:
\begin{align*}
	& \bra{\mu'''} R_{i}R_{i-1}p_{i}^{(\nu)} \ket{\mu} \\
	& = \sum_{\nu',\nu''\in\mathcal{I}} \sum_{\mu',\mu''\in\mathcal{B}_{\tilde{\lambda}}} R^{\lambda\lambda}_{\nu''}R^{\lambda\lambda}_{\nu'} \bra{\mu'''} p_{i}^{(\nu'')}\ket{\mu''}\bra{\mu''}p_{i-1}^{(\nu')}\ket{\mu'}\bra{\mu'}p_{i}^{(\nu)} \ket{\mu}\\
	& = \sum_{\upsilon_{3},\nu',\nu''\in\mathcal{I}} \left[ \prod_{j\neq i-1,i} \delta_{\mu_{j}}^{\mu_{j}'''} \right] 
	R^{\lambda\lambda}_{\nu''}R^{\lambda\lambda}_{\nu'}
	\left(\bar{F}^{\lambda\lambda\lambda}_{\upsilon_{3}}\right)^{\nu}_{\nu'} 
	\left(F^{\lambda\lambda\lambda}_{\upsilon_{3}}\right)^{\nu'}_{\nu''} 
	\left(\bar{F}^{\mu_{i-1}'''\lambda\lambda}_{\mu_{i+1}}\right)_{\mu_{i}'''}^{\nu''} 
	\left(F^{\mu_{i-1}\lambda\lambda}_{\mu_{i+1}}\right)^{\mu_{i}}_{\nu} 
	\left(F^{\mu_{i-2}\lambda\nu}_{\mu_{i+1}}\right)^{\mu_{i-1}}_{\upsilon_{3}} 
	\left(\bar{F}^{\mu_{i-2}\lambda\nu''}_{\mu_{i+1}}\right)^{\upsilon_{3}}_{\mu_{i-1}'''} \\
	& = R^{\nu\lambda}_{\phi_{\nu}(\lambda)} \sum_{\upsilon_{3},\nu',\nu''\in\mathcal{I}} \left[ \prod_{j\neq i-1,i} \delta_{\mu_{j}}^{\mu_{j}'''} \right] 
	\left(F^{\lambda\lambda\lambda}_{\upsilon_{3}}\right)^{\nu}_{\nu''} 
	\left(\bar{F}^{\mu_{i-1}'''\lambda\lambda}_{\mu_{i+1}}\right)_{\mu_{i}'''}^{\nu''} 
	\left(F^{\mu_{i-1}\lambda\lambda}_{\mu_{i+1}}\right)^{\mu_{i}}_{\nu} 
	\left(F^{\mu_{i-2}\lambda\nu}_{\mu_{i+1}}\right)^{\mu_{i-1}}_{\upsilon_{3}} 
	\left(\bar{F}^{\mu_{i-2}\lambda\nu''}_{\mu_{i+1}}\right)^{\upsilon_{3}}_{\mu_{i-1}'''} \\
	& = R^{\nu\lambda}_{\phi_{\nu}(\lambda)} \sum_{\upsilon_{3}\in\mathcal{I}} \left[ \prod_{j\neq i-1,i} \delta_{\mu_{j}}^{\mu_{j}'''} \right] 
	\left(\bar{F}^{\mu_{i-2}\lambda\lambda}_{\mu_{i}'''}\right)^{\nu}_{\mu_{i-1}'''}
	\left(F^{\mu_{i-1}\lambda\lambda}_{\mu_{i+1}}\right)^{\mu_{i}}_{\nu} 
	\left(\bar{F}^{\mu_{i-2}\nu\lambda}_{\mu_{i+1}}\right)^{\upsilon_{3}}_{\mu_{i}'''} 
	\left(F^{\mu_{i-2}\lambda\nu}_{\mu_{i+1}}\right)^{\mu_{i-1}}_{\upsilon_{3}} \\
	& = c^{-1} R^{\nu\lambda}_{\phi_{\nu}(\lambda)} \sum_{\upsilon_{3}\in\mathcal{I}} \left[ \prod_{j\neq i-1,i} \delta_{\mu_{j}}^{\mu_{j}'''} \right] 
	\delta_{\upsilon_{4}}^{\mu_{i}'''} \left(F^{\lambda\lambda\lambda}_{\upsilon_{3}}\right)^{\nu}_{\nu} 
	\left(\bar{F}^{\mu_{i-2}\lambda\lambda}_{\mu_{i}'''}\right)^{\nu}_{\mu_{i-1}'''}
	\left(F^{\mu_{i-1}\lambda\lambda}_{\mu_{i+1}}\right)^{\mu_{i}}_{\nu} \left(\bar{F}^{\mu_{i-1}\lambda\lambda}_{\mu_{i+1}}\right)^{\nu}_{\upsilon_{4}} \left(F^{\mu_{i-2}\lambda\lambda}_{\upsilon_{4}}\right)^{\mu_{i-1}}_{\nu} \\
	& = c^{-1} \left[	R^{\nu\lambda}_{\phi_{\nu}(\lambda)} \left(F^{\lambda\lambda\lambda}_{\phi_{\nu}(\lambda)}\right)^{\nu}_{\nu} \right] \times \bra{\mu'''} p_{i-1}^{(\nu)}p_{i}^{(\nu)} \ket{\mu}
\end{align*}
where we have used Calculation \ref{calcpppV2} and the relations
\begin{align*}
	\sum_{\nu'\in\mathcal{I}} \left(\bar{F}^{\lambda\lambda\lambda}_{\upsilon_{3}}\right)^{\nu}_{\nu'} R^{\lambda\lambda}_{\nu'} \left(F^{\lambda\lambda\lambda}_{\upsilon_{3}}\right)^{\nu'}_{\nu''} R^{\lambda\lambda}_{\nu''} & = R^{\nu\lambda}_{\upsilon_{3}} \left(F^{\lambda\lambda\lambda}_{\upsilon_{3}}\right)^{\nu}_{\nu''} \\
	\sum_{\nu''} \left(F^{\lambda\lambda\lambda}_{\upsilon_{3}}\right)_{\nu''}^{\nu} \left(\bar{F}^{\mu_{i-2}\lambda\nu''}_{\mu_{i+1}}\right)^{\upsilon_{3}}_{\mu_{i-1}'''} \left(\bar{F}^{\mu_{i-1}'''\lambda\lambda}_{\mu_{i+1}}\right)^{\nu''}_{\mu_{i}'''} & = \left(\bar{F}^{\mu_{i-2}\nu\lambda}_{\mu_{i+1}}\right)^{\upsilon_{3}}_{\mu_{i}'''} \left(\bar{F}^{\mu_{i-2}\lambda\lambda}_{\mu_{i}'''}\right)^{\nu}_{\mu_{i-1}'''} \\
	\sum_{\upsilon_{4}} \left(\bar{F}^{\mu_{i-1}\lambda\lambda}_{\mu_{i+1}}\right)^{\nu}_{\upsilon_{4}} \left(F^{\mu_{i-2}\lambda\lambda}_{\upsilon_{4}}\right)^{\mu_{i-1}}_{\nu} \left(F^{\mu_{i-2}\nu\lambda}_{\mu_{i+1}}\right)^{\upsilon_{4}}_{\upsilon_{3}} & = \left(F^{\mu_{i-2}\lambda\nu}_{\mu_{i+1}}\right)^{\mu_{i-1}}_{\upsilon_{3}} \left(\bar{F}^{\lambda\lambda\lambda}_{\upsilon_{3}}\right)^{\nu}_{\nu} \\
	\left(F^{\mu_{i-2}\nu\lambda}_{\mu_{i+1}}\right)^{\upsilon_{4}}_{\phi_{\nu}(\lambda)} \left(\bar{F}^{\mu_{i-2}\nu\lambda}_{\mu_{i+1}}\right)^{\phi_{\nu}(\lambda)}_{\mu_{i}'''} & = \delta_{\upsilon_{4}}^{\mu_{i}'''} (1-\delta_{N_{\mu_{i-2}\nu}^{\upsilon_{4}}}^{0}) (1-\delta_{N_{\upsilon_{4}\lambda}^{\mu_{i+1}}}^{0})
\end{align*}
This proves the top relation. We now consider the second relation
\begin{align*}
	& \bra{\mu'''} R_{i}^{-1}R_{i-1}^{-1}p_{i}^{(\nu)} \ket{\mu} \\
	& = \sum_{\nu',\nu''\in\mathcal{I}} \sum_{\mu',\mu''\in\mathcal{B}_{\tilde{\lambda}}} R^{\lambda\lambda}_{\nu''}R^{\lambda\lambda}_{\nu'} \bra{\mu'''} p_{i}^{(\nu'')}\ket{\mu''}\bra{\mu''}p_{i-1}^{(\nu')}\ket{\mu'}\bra{\mu'}p_{i}^{(\nu)} \ket{\mu}\\
	& = \sum_{\upsilon_{3},\nu',\nu''\in\mathcal{I}} \left[ \prod_{j\neq i-1,i} \delta_{\mu_{j}}^{\mu_{j}'''} \right] 
	\bar{R}^{\lambda\lambda}_{\nu''}\bar{R}^{\lambda\lambda}_{\nu'}
	\left(\bar{F}^{\lambda\lambda\lambda}_{\upsilon_{3}}\right)^{\nu}_{\nu'} 
	\left(F^{\lambda\lambda\lambda}_{\upsilon_{3}}\right)^{\nu'}_{\nu''} 
	\left(\bar{F}^{\mu_{i-1}'''\lambda\lambda}_{\mu_{i+1}}\right)_{\mu_{i}'''}^{\nu''} 
	\left(F^{\mu_{i-1}\lambda\lambda}_{\mu_{i+1}}\right)^{\mu_{i}}_{\nu} 
	\left(F^{\mu_{i-2}\lambda\nu}_{\mu_{i+1}}\right)^{\mu_{i-1}}_{\upsilon_{3}} 
	\left(\bar{F}^{\mu_{i-2}\lambda\nu''}_{\mu_{i+1}}\right)^{\upsilon_{3}}_{\mu_{i-1}'''} \\
	& = \bar{R}^{\nu\lambda}_{\phi_{\nu}(\lambda)} \sum_{\upsilon_{3},\nu''\in\mathcal{I}} \left[ \prod_{j\neq i-1,i} \delta_{\mu_{j}}^{\mu_{j}'''} \right] 
	\left(F^{\lambda\lambda\lambda}_{\upsilon_{3}}\right)^{\nu}_{\nu''}
	\left(\bar{F}^{\mu_{i-1}'''\lambda\lambda}_{\mu_{i+1}}\right)_{\mu_{i}'''}^{\nu''} 
	\left(F^{\mu_{i-1}\lambda\lambda}_{\mu_{i+1}}\right)^{\mu_{i}}_{\nu} 
	\left(F^{\mu_{i-2}\lambda\nu}_{\mu_{i+1}}\right)^{\mu_{i-1}}_{\upsilon_{3}} 
	\left(\bar{F}^{\mu_{i-2}\lambda\nu''}_{\mu_{i+1}}\right)^{\upsilon_{3}}_{\mu_{i-1}'''} \\
	& = c^{-1} \left[	\bar{R}^{\nu\lambda}_{\phi_{\nu}(\lambda)} \left(F^{\lambda\lambda\lambda}_{\phi_{\nu}(\lambda)}\right)^{\nu}_{\nu} \right] \times \bra{\mu'''} p_{i-1}^{(\nu)}p_{i}^{(\nu)} \ket{\mu}
\end{align*}
where we have used Calculation \ref{calcpppV2} and the relationship written previously as well as
\begin{align*}
	\sum_{\nu'\in\mathcal{I}} \left(\bar{F}^{\lambda\lambda\lambda}_{\upsilon_{3}}\right)^{\nu}_{\nu'} \bar{R}^{\lambda\lambda}_{\nu'} \left(F^{\lambda\lambda\lambda}_{\upsilon_{3}}\right)^{\nu'}_{\nu''} \bar{R}^{\lambda\lambda}_{\nu''} & =  \bar{R}^{\nu\lambda}_{\upsilon_{3}} \left(F^{\lambda\lambda\lambda}_{\upsilon_{3}}\right)^{\nu}_{\nu''}
\end{align*}
This proves the second relationship. We now consider the third relation
\begin{align*}
	& \bra{\mu'''} p_{i-1}^{(\nu)}R_{i}R_{i-1} \ket{\mu} \\
	& = \sum_{\nu',\nu''\in\mathcal{I}} \sum_{\mu',\mu''\in\mathcal{B}_{\tilde{\lambda}}} R^{\lambda\lambda}_{\nu''}R^{\lambda\lambda}_{\nu'} \bra{\mu'''} p_{i-1}^{(\nu)}\ket{\mu''}\bra{\mu''}p_{i}^{(\nu')}\ket{\mu'}\bra{\mu'}p_{i-1}^{(\nu'')} \ket{\mu}\\
	& = \sum_{\upsilon_{3},\nu',\nu''\in\mathcal{I}} \left[ \prod_{j\neq i-1,i} \delta_{\mu_{j}}^{\mu_{j}'''} \right] R^{\lambda\lambda}_{\nu''}R^{\lambda\lambda}_{\nu'}
	\left(F^{\lambda\lambda\lambda}_{\upsilon_{3}}\right)^{\nu''}_{\nu'} 
	\left(\bar{F}^{\lambda\lambda\lambda}_{\upsilon_{3}}\right)^{\nu'}_{\nu}
	\left(\bar{F}^{\mu_{i-2}\lambda\lambda}_{\mu_{i}'''}\right)_{\mu_{i-1}'''}^{\nu} 
	\left(F^{\mu_{i-2}\lambda\lambda}_{\mu_{i}}\right)^{\mu_{i-1}}_{\nu''} 
	\left(F^{\mu_{i-2}\nu''\lambda}_{\mu_{i+1}}\right)^{\mu_{i}}_{\upsilon_{3}} 
	\left(\bar{F}^{\mu_{i-2}\nu\lambda}_{\mu_{i+1}}\right)^{\upsilon_{3}}_{\mu_{i}'''} \\
	& = R^{\nu\lambda}_{\phi_{\nu}(\lambda)} \sum_{\upsilon_{3},\nu''\in\mathcal{I}} \left[ \prod_{j\neq i-1,i} \delta_{\mu_{j}}^{\mu_{j}'''} \right] 
	\left(F^{\lambda\lambda\lambda}_{\upsilon_{3}}\right)^{\nu''}_{\nu}
	\left(\bar{F}^{\mu_{i-2}\lambda\lambda}_{\mu_{i}'''}\right)_{\mu_{i-1}'''}^{\nu} 
	\left(F^{\mu_{i-2}\lambda\lambda}_{\mu_{i}}\right)^{\mu_{i-1}}_{\nu''} 
	\left(F^{\mu_{i-2}\nu''\lambda}_{\mu_{i+1}}\right)^{\mu_{i}}_{\upsilon_{3}} 
	\left(\bar{F}^{\mu_{i-2}\nu\lambda}_{\mu_{i+1}}\right)^{\upsilon_{3}}_{\mu_{i}'''} \\
	& = R^{\nu\lambda}_{\phi_{\nu}(\lambda)} \sum_{\upsilon_{3}\in\mathcal{I}} \left[ \prod_{j\neq i-1,i} \delta_{\mu_{j}}^{\mu_{j}'''} \right] 
	\left(\bar{F}^{\mu_{i-2}\lambda\lambda}_{\mu_{i}'''}\right)_{\mu_{i-1}'''}^{\nu} 
	\left(\bar{F}^{\mu_{i-2}\nu\lambda}_{\mu_{i+1}}\right)^{\upsilon_{3}}_{\mu_{i}'''}  
	\left(F^{\mu_{i-1}\lambda\lambda}_{\mu_{i+1}}\right)^{\mu_{i}}_{\nu} 
	\left(F^{\mu_{i-2}\lambda\nu}_{\mu_{i+1}}\right)^{\mu_{i-1}}_{\upsilon_{3}}\\
	& = c^{-1} R^{\nu\lambda}_{\phi_{\nu}(\lambda)} \sum_{\upsilon_{3},\upsilon_{4}\in\mathcal{I}} \left[ \prod_{j\neq i-1,i} \delta_{\mu_{j}}^{\mu_{j}'''} \right] \delta_{\upsilon_{4}}^{\mu_{i-1}} 
	\left(\bar{F}^{\mu_{i-2}\lambda\lambda}_{\mu_{i}'''}\right)_{\mu_{i-1}'''}^{\nu} 
	\left(F^{\mu_{i-1}\lambda\lambda}_{\mu_{i+1}}\right)^{\mu_{i}}_{\nu} 
	\left(F^{\lambda\lambda\lambda}_{\upsilon_{3}}\right)^{\nu}_{\nu} 
	\left(\bar{F}^{\upsilon_{4}\lambda\lambda}_{\mu_{i+1}}\right)^{\nu}_{\mu_{i}'''} 
	\left(F^{\mu_{i-2}\lambda\lambda}_{\mu_{i}'''}\right)^{\upsilon_{4}}_{\nu} \\
	& = c^{-1} \left[ R^{\nu\lambda}_{\phi_{\nu}(\lambda)} \left(F^{\lambda\lambda\lambda}_{\phi_{\nu}(\lambda)}\right)^{\nu}_{\nu} \right]\bra{\mu'''} p_{i-1}^{(\nu)}p_{i}^{(\nu)} \ket{\mu}
\end{align*}
where we have used Calculation \ref{calcpppV1} and the relationship written previously as well as
\begin{align*}
	\sum_{\nu'\in\mathcal{I}}  R^{\lambda\lambda}_{\nu''}\left(F^{\lambda\lambda\lambda}_{\upsilon_{3}}\right)^{\nu''}_{\nu'} R^{\lambda\lambda}_{\nu'} \left(\bar{F}^{\lambda\lambda\lambda}_{\upsilon_{3}}\right)^{\nu'}_{\nu} & = \left(F^{\lambda\lambda\lambda}_{\upsilon_{3}}\right)^{\nu''}_{\nu} R^{\nu\lambda}_{\upsilon_{3}} \\
	\sum_{\nu''\in\mathcal{I}} \left(F^{\mu_{i-2}\lambda\lambda}_{\mu_{i}}\right)^{\mu_{i-1}}_{\nu''} \left(F^{\mu_{i-2}\nu''\lambda}_{\mu_{i+1}}\right)^{\mu_{i}}_{\upsilon_{3}} \left(F^{\lambda\lambda\lambda}_{k}\right)^{\nu''}_{\nu} & = \left(F^{\mu_{i-1}\lambda\lambda}_{\mu_{i+1}}\right)^{\mu_{i}}_{\nu} \left(F^{\mu_{i-2}\lambda\nu}_{\mu_{i+1}}\right)^{\mu_{i-1}}_{\upsilon_{3}} \\
	\sum_{\upsilon_{4}} \left(\bar{F}^{\mu_{i-2}\lambda\nu}_{\mu_{i+1}}\right)^{\upsilon_{3}}_{\upsilon_{4}} \left(\bar{F}^{\upsilon_{4}\lambda\lambda}_{\mu_{i+1}}\right)^{\nu}_{\mu_{i}'''} \left(F^{\mu_{i-2}\lambda\lambda}_{\mu_{i}'''}\right)^{\upsilon_{4}}_{\nu} & = \left(\bar{F}^{\mu_{i-2}\nu\lambda}_{\mu_{i+1}}\right)^{\upsilon_{3}}_{\mu_{i}'''} \left(\bar{F}^{\lambda\lambda\lambda}_{\upsilon_{3}}\right)^{\nu}_{\nu}  \\
	\left(F^{\mu_{i-2}\lambda\nu}_{\mu_{i+1}}\right)^{\mu_{i-1}}_{\phi_{\nu}'(\lambda)} \left(\bar{F}^{\mu_{i-2}\lambda\nu}_{\mu_{i+1}}\right)^{\phi_{\nu}'(\lambda)}_{\upsilon_{4}} & = \delta_{\upsilon_{4}}^{\mu_{i-1}} (1 - \delta_{N_{\mu_{i-1}\nu}^{\mu_{i+1}}}^{0}) (1 - \delta_{N_{\mu_{i-2}\lambda}^{\mu_{i-1}}}^{0})
\end{align*}
This proves the third relationship. We now consider the fourth
\begin{align*}
	& \bra{\mu'''} p_{i-1}^{(\nu)}R_{i}^{-1}R_{i-1}^{-1} \ket{\mu} \\
	& = \sum_{\nu',\nu''\in\mathcal{I}} \sum_{\mu',\mu''\in\mathcal{B}_{\tilde{\lambda}}} \bar{R}^{\lambda\lambda}_{\nu''} \bar{R}^{\lambda\lambda}_{\nu'} \bra{\mu'''} p_{i-1}^{(\nu)}\ket{\mu''}\bra{\mu''}p_{i}^{(\nu')}\ket{\mu'}\bra{\mu'}p_{i-1}^{(\nu'')} \ket{\mu}\\
	& = \sum_{\upsilon_{3},\nu',\nu''\in\mathcal{I}} \left[ \prod_{j\neq i-1,i} \delta_{\mu_{j}}^{\mu_{j}'''} \right] 
	\bar{R}^{\lambda\lambda}_{\nu''} \bar{R}^{\lambda\lambda}_{\nu'}
	\left(F^{\lambda\lambda\lambda}_{\upsilon_{3}}\right)^{\nu''}_{\nu'} 
	\left(\bar{F}^{\lambda\lambda\lambda}_{\upsilon_{3}}\right)^{\nu'}_{\nu}
	\left(\bar{F}^{\mu_{i-2}\lambda\lambda}_{\mu_{i}'''}\right)_{\mu_{i-1}'''}^{\nu} 
	\left(F^{\mu_{i-2}\lambda\lambda}_{\mu_{i}}\right)^{\mu_{i-1}}_{\nu''} 
	\left(F^{\mu_{i-2}\nu''\lambda}_{\mu_{i+1}}\right)^{\mu_{i}}_{\upsilon_{3}} 
	\left(\bar{F}^{\mu_{i-2}\nu\lambda}_{\mu_{i+1}}\right)^{\upsilon_{3}}_{\mu_{i}'''} \\
	& = \bar{R}^{\nu\lambda}_{\phi_{\nu}(\lambda)}  \sum_{\upsilon_{3},\nu''\in\mathcal{I}} \left[ \prod_{j\neq i-1,i} \delta_{\mu_{j}}^{\mu_{j}'''} \right] 
	\left(F^{\lambda\lambda\lambda}_{\upsilon_{3}}\right)^{\nu''}_{\nu} 
	\left(\bar{F}^{\mu_{i-2}\lambda\lambda}_{\mu_{i}'''}\right)_{\mu_{i-1}'''}^{\nu} 
	\left(F^{\mu_{i-2}\lambda\lambda}_{\mu_{i}}\right)^{\mu_{i-1}}_{\nu''} 
	\left(F^{\mu_{i-2}\nu''\lambda}_{\mu_{i+1}}\right)^{\mu_{i}}_{\upsilon_{3}} 
	\left(\bar{F}^{\mu_{i-2}\nu\lambda}_{\mu_{i+1}}\right)^{\upsilon_{3}}_{\mu_{i}'''} \\
	& = c^{-1} \left[ \bar{R}^{\nu\lambda}_{\phi_{\nu}(\lambda)} \left(F^{\lambda\lambda\lambda}_{\phi_{\nu}(\lambda)}\right)^{\nu}_{\nu} \right] \bra{\mu'''} p_{i-1}^{(\nu)}p_{i}^{(\nu)} \ket{\mu}
\end{align*}
where we have used Calculation \ref{calcpppV1} and the relationship written previously as well as
\begin{align*}
 \sum_{\nu'\in\mathcal{I}} \bar{R}^{\lambda\lambda}_{\nu''}\left(F^{\lambda\lambda\lambda}_{\upsilon_{3}}\right)^{\nu''}_{\nu'} \bar{R}^{\lambda\lambda}_{\nu'} \left(\bar{F}^{\lambda\lambda\lambda}_{\upsilon_{3}}\right)^{\nu'}_{\nu} & = \left(F^{\lambda\lambda\lambda}_{\upsilon_{3}}\right)^{\nu''}_{\nu} \bar{R}^{\nu\lambda}_{\upsilon_{3}}.
\end{align*}
Thus the fourth relation is proven.
\end{calc}

\begin{calc} \label{calcRRppp3}
We first determine special cases of the hexagon equation:
\begin{align*}
	\sum_{\upsilon\in\mathcal{I}} \left(\bar{F}^{\lambda\lambda\lambda}_{\phi_{\nu}(\lambda)}\right)^{\nu}_{\upsilon} R^{\lambda\lambda}_{\upsilon}\left(F^{\lambda\lambda\lambda}_{\phi_{\nu}(\lambda)}\right)^{\upsilon}_{\nu} R^{\lambda\lambda}_{\nu} = R^{\nu\lambda}_{\phi_{\nu}(\lambda)} \left(F^{\lambda\lambda\lambda}_{\phi_{\nu}(\lambda)}\right)^{\nu}_{\nu}, \\
	\sum_{\upsilon\in\mathcal{I}} R^{\lambda\lambda}_{\nu} \left(\bar{F}^{\lambda\lambda\lambda}_{\phi_{\nu}(\lambda)}\right)^{\nu}_{\upsilon} R^{\lambda\lambda}_{\upsilon} \left(F^{\lambda\lambda\lambda}_{\phi_{\nu}(\lambda)}\right)^{\upsilon}_{\nu} = \left(\bar{F}^{\lambda\lambda\lambda}_{\phi_{\nu}(\lambda)}\right)^{\nu}_{\nu} R^{\lambda\nu}_{\phi_{\nu}(\lambda)}, \\
	\sum_{\upsilon\in\mathcal{I}} \left(\bar{F}^{\lambda\lambda\lambda}_{\phi_{\nu}(\lambda)}\right)^{\nu}_{\upsilon} \bar{R}^{\lambda\lambda}_{\upsilon}\left(F^{\lambda\lambda\lambda}_{\phi_{\nu}(\lambda)}\right)^{\upsilon}_{\nu} \bar{R}^{\lambda\lambda}_{\nu} = \bar{R}^{\nu\lambda}_{\phi_{\nu}(\lambda)} \left(F^{\lambda\lambda\lambda}_{\phi_{\nu}(\lambda)}\right)^{\nu}_{\nu}, \\
	\sum_{e\in\mathcal{I}}  \bar{R}^{\lambda\lambda}_{\nu} \left(\bar{F}^{\lambda\lambda\lambda}_{\phi_{\nu}(\lambda)}\right)^{\nu}_{\upsilon} \bar{R}^{\lambda\lambda}_{\upsilon} \left(F^{\lambda\lambda\lambda}_{\phi_{\nu}(\lambda)}\right)^{\upsilon}_{\nu} = \left(\bar{F}^{\lambda\lambda\lambda}_{\phi_{\nu}(\lambda)}\right)^{\nu}_{\nu} \bar{R}^{\lambda\nu}_{\phi_{\nu}(\lambda)}. 
\end{align*}
Using Calculations \ref{calcRRppp1} and \ref{calcRRppp2} it follows that if $\nu$ acts one-dimensionally on $\lambda$ and there exists $c\in\R^{\times}$ such that
\begin{align*}
  c = \left(F^{\lambda\lambda\lambda}_{\phi_{\nu}(\lambda)}\right)^{\nu}_{\nu} \left(\bar{F}^{\lambda\lambda\lambda}_{\phi_{\nu}(\lambda)}\right)^{\nu}_{\nu},
\end{align*}
then
\begin{align*}
	R_{i}R_{i\pm1}p_{i}^{(\nu)} 
	= p_{i\pm1}^{(\nu)}R_{i}R_{i\pm1} 
	= c^{-1} \left[R^{\lambda\lambda}_{\nu} \sum_{\upsilon\in\mathcal{I}} \left(F^{\lambda\lambda\lambda}_{\phi_{\nu}(\lambda)}\right)^{\upsilon}_{\nu}  R^{\lambda\lambda}_{\upsilon}  \left(\bar{F}^{\lambda\lambda\lambda}_{\phi_{\nu}(\lambda)}\right)^{\nu}_{\upsilon} \right] p_{i\pm1}^{(\nu)}p_{i}^{(\nu)} \\
	R_{i}^{-1}R_{i\pm1}^{-1}p_{i}^{(\nu)} 
	= p_{i\pm1}^{(\nu)}R_{i}^{-1}R_{i\pm1}^{-1} 
	= c^{-1} \left[\bar{R}^{\lambda\lambda}_{\nu} \sum_{\upsilon\in\mathcal{I}} \left(F^{\lambda\lambda\lambda}_{\phi_{\nu}(\lambda)}\right)^{\upsilon}_{\nu}  R^{\lambda\lambda}_{\upsilon}  \left(\bar{F}^{\lambda\lambda\lambda}_{\phi_{\nu}(\lambda)}\right)^{\nu}_{\upsilon} \right] p_{i\pm1}^{(\nu)}p_{i}^{(\nu)}
\end{align*}
\end{calc}

\end{document}